\documentclass[11pt,reqno]{amsart}
\usepackage{a4wide}
\usepackage{amssymb}
\usepackage{amsthm}
\usepackage{amsmath}
\usepackage{amscd}
\usepackage{verbatim,url}
\usepackage{color,cancel}
\usepackage[mathscr]{eucal}
\allowdisplaybreaks

\newtheorem{theorem}{Theorem}
\newtheorem{lemma}[theorem]{Lemma}
\newtheorem{corollary}[theorem]{Corollary}

\newtheorem{proposition}[theorem]{Proposition}

\theoremstyle{remark}

\numberwithin{theorem}{section} \numberwithin{equation}{section}
\numberwithin{figure}{section}

\begin{document}
\title[]{Bailey pairs and strange identities}

\author{Jeremy Lovejoy}
\address{CNRS,
Universit\'e de Paris,
B\^atiment Sophie Germain, Case courier 7014,
8 Place Aur\'elie Nemours,
FRANCE 75205 Paris Cedex 13, FRANCE}
\email{lovejoy@math.cnrs.fr}

\date{\today}
\thanks{}

\begin{abstract}
Zagier introduced the term ``strange identity" to describe an asymptotic relation between a certain $q$-hypergeometric series and a partial theta function at roots of unity.   We show that behind Zagier's strange identity lies a statement about Bailey pairs.   Using the iterative machinery of Bailey pairs then leads to many families of multisum strange identities, including Hikami's generalization of Zagier's identity.   
\end{abstract}

\maketitle

\section{Introduction and Statement of Results}
A \emph{Bailey pair} relative to $(a,q)$ is a pair of sequences $(\alpha_n,\beta_n)_{n \geq 0}$ satisfying
\begin{equation} \label{pairdef}
\beta_n = \sum_{k=0}^n \frac{\alpha_k}{(q)_{n-k}(aq)_{n+k}}. 
\end{equation} 
%or equivalently 
%\begin{equation} \label{pairdefbis}
%\alpha_n = \frac{(1-aq^{2n})}{1-a} \sum_{j=0}^n \frac{(a)_{n+j}(-1)^{n-j}q^{\binom{n-j}{2}}}{(q)_{n-j}}\beta_j.
%\end{equation}
Here we have used the standard $q$-hypergeometric notation, 
\begin{equation*}
(a)_n = (a;q)_n = \prod_{k=1}^{n} (1-aq^{k-1}),
\end{equation*}
defined for integers $n \geq 0$ and in the limit as $n \to \infty$.    Bailey pairs are one of the principal structural elements in the theory of $q$-hypergeometric series.   Much of their power comes from the fact that Bailey pairs give rise to new Bailey pairs, and they do so in many different ways.    This leads to an iterative ``machinery" that produces infinite families of identities starting from a single identity.   For example, once one understands how to use Bailey pairs to prove the Rogers-Ramanujan identities,
\begin{align*}
\sum_{n \geq 0} \frac{q^{n^2}}{(q)_n} &= \prod_{n \equiv 1,4 \pmod{5}} \frac{1}{1-q^n} \\
\intertext{and}
\sum_{n \geq 0} \frac{q^{n^2+n}}{(q)_n} &= \prod_{n \equiv 2,3 \pmod{5}} \frac{1}{1-q^n},
\end{align*}   
the Bailey machinery produces the entire family of Andrews-Gordon identities,
\begin{equation} \label{AG}
\sum_{n_{k-1} \geq \cdots \geq n_1 \geq 0} \frac{q^{n_1^2 + \cdots + n_{k-1}^2 + n_1 + \cdots + n_{k-i}}}{(q)_{n_{k-1}-n_{k-2}} \cdots (q)_{n_2-n_1}(q)_{n_1}} = \prod_{n \not \equiv 0, \pm i \pmod{2k+1}} \frac{1}{1-q^n}.
\end{equation}
Here $1 \leq i \leq k$.
As a bonus, such identities often arise naturally in number theory, combinatorics, algebra, physics, or knot theory.   For more on Bailey pairs and their applications, see \cite{An1,An2,An3,Mc1,Wa1}

Here we consider the role of Bailey pairs in certain $q$-series identities which are not really identities at all.   The term \emph{strange identity} was first used by Zagier \cite{Za1} to describe the fact that
\begin{equation} \label{Zagierstrange}
\sum_{n \geq 0} (q)_n ``=" - \frac{1}{2} \sum_{n \geq 1} n \left( \frac{12}{n} \right)q^{(n^2-1)/24}.
\end{equation}
%where $\chi_{12}^{(0)}(n)$ is the even periodic function modulo $12$ defined by
%\begin{equation}
%\chi_{12}^{(a)}(n) = 
%\begin{cases}
%1, &\text{if $n \equiv \pm 1 \pmod{12}$}, \\
%-1, &\text{if $n \equiv \pm 5 \pmod{12}$}, \\
%0, &\text{otherwise},
%\end{cases}
%\end{equation}
What the symbol $``="$ means here is not that the identity holds in any classical sense, but that both sides agree to infinite order at any root of unity $\zeta$.    Somewhat more specifically, replacing $q$ by $\zeta e^{-t}$ and letting $t \to 0^+$, the right-hand side has an asymptotic expansion as a power series in $t$, and this power series is given by the left-hand side at $q=\zeta e^{-t}$.  See \cite{Ah-Ki-Lo1} or  \cite{Za1} for a much more detailed discussion.  

Hikami \cite{Hi1}  gave an elegant generalization of Zagier's strange identity.   To state it, 
%recall the usual $q$-rising factorial, defined for integers $n \geq 0$ and in the limit as $n \to \infty$ by
%\begin{equation}
%(a)_n = (a;q)_n = \prod_{j=1}^{n} (1-aq^{j-1}),
%\end{equation}
recall the $q$-binomial coefficient defined by
\begin{equation}\label{qbincoeff}
\begin{bmatrix} n \\ k \end{bmatrix} =  \begin{bmatrix} n \\ k \end{bmatrix}_q = 
\begin{cases}
\frac{(q)_n}{(q)_{n-k}(q)_k}, &\text{if $0 \leq k \leq n$}, \\
0, &\text{otherwise}.
\end{cases}
\end{equation}
Then for $0 \leq a \leq k-1$ Hikami showed that 
\begin{equation} \label{Hikamistrange}
\begin{aligned}
\sum_{n_1,\dots,n_k \geq 0}& (q)_{n_k}q^{n_1^2+ \cdots +n_{k-1}^2 + n_{a+1} + \cdots + n_{k-1}} \prod_{i=1}^{k-1} \begin{bmatrix} n_{i+1} + \delta_{i,a} \\ n_i \end{bmatrix}  \\
&``=" -\frac{1}{2} \sum_{n \geq 0} n\chi_{8k+4}^{(a)}(n)q^{\frac{n^2 - (2k-2a-1)^2}{8(2k+1)}},
\end{aligned}
\end{equation}
where $\chi_{8k+4}^{(a)}(n)$ is the even periodic function modulo $8k+4$ defined by
\begin{equation} \label{Hikamichar}
\chi_{8k+4}^{(a)}(n) = 
\begin{cases}
1, &\text{if $n \equiv 2k-2a-1$ or $6k+2a+5 \pmod{8k+4}$}, \\
-1, &\text{if $n \equiv 2k+2a+3$ or $6k-2a+1 \pmod{8k+4}$}, \\
0, &\text{otherwise}.
\end{cases}
\end{equation}

Although it is not our primary focus here, we note that strange identities have a number of interesting applications.    They can be used to establish the quantum modularity of (appropriate normalizations of) the corresponding $q$-series, to find formulas for the values of these series at roots of unity, and to give $q$-hypergeometric generating functions for values of certain $L$-functions at negative integers.   See \cite{An-Ji-On1,Br-Ro1,Go-Os1,Hi1,Za1}, for example.    Strange identities also play a key role in the study of congruences and asymptotics for the coefficients of the corresponding $q$-series at $q=1-q$ \cite{Ah-Ki1,Ah-Ki-Lo1,An-Se1,Bietal,Ga1,Go1,Gu-Ke-Ro1,St1}.

Hikami's proof of the strange identities \eqref{Hikamistrange} involved some long and impressive computations using $q$-difference equations \cite{Hi1}.   In this paper we show how these identities can be understood in the context of Bailey pairs.   In addition to providing a streamlined proof, this leads us to the discovery of many more families of strange identities.   Theorem \ref{main} below contains a selection of some of the nicest of these.   Note that unlike \eqref{Zagierstrange} and \eqref{Hikamistrange}, these strange identities only hold at appropriate subsets of roots of unity -- namely, those which do not cause zeros in the denominators of the $q$-hypergeometric series.  
\begin{theorem} \label{main}
Assume that $0 \leq a \leq k-1$.   
\begin{itemize}
\item[($i$)] Let $\chi_{4k}(n)$ be the even periodic function modulo $4k$ defined by
\begin{equation} \label{char1}
\chi_{4k}(n) = 
\begin{cases}
1, &\text{if $n \equiv k-1$ or $3k+1 \pmod{4k}$}, \\
-1, &\text{if $n \equiv k+1$ or $3k-1 \pmod{4k}$}, \\
0, &\text{otherwise}.
\end{cases}
\end{equation}
Then 
\begin{equation} \label{family1}
\begin{aligned}
\sum_{n_1,\dots,n_k \geq 0}& (q)_{n_k}\frac{q^{n_1^2+ \cdots +n_{k-1}^2 + n_{1} + \cdots + n_{k-1}}}{(-q)_{n_1}} \prod_{i=1}^{k-1} \begin{bmatrix} n_{i+1} \\ n_i \end{bmatrix} \\
&``=" -(1+\delta_{k,1})\sum_{n \geq 0} n\chi_{4k}(n)q^{\frac{n^2 - (k-1)^2}{4k}}
\end{aligned}
\end{equation}
and
\begin{equation} \label{family2}
\begin{aligned}
\sum_{n_1,\dots,n_k \geq 0}& (q^2;q^2)_{n_k}\frac{q^{2n_1^2+2n_1+ \cdots +2n_{k-1}^2 + 2n_{k-1}}(q;q^2)_{n_1}}{(-q)_{2n_1+1}} \prod_{i=1}^{k-1} \begin{bmatrix} n_{i+1} \\ n_i \end{bmatrix}_{q^2} \\
&``=" -(1+\delta_{k,1})\frac{1}{2}\sum_{n \geq 0} n\chi_{8k-4}(n)q^{\frac{n^2 - (2k-2)^2}{8k-4}}.
\end{aligned}
\end{equation}
 \\
\item[($ii$)] Let $\chi_{8k}^{(a)}(n)$ be the even periodic function modulo $8k$ defined by
\begin{equation} \label{char2}
\chi_{8k}^{(a)}(n) = 
\begin{cases}
1, &\text{if $n \equiv 2k-2a-1$ or $6k+2a+1 \pmod{8k}$}, \\
-1, &\text{if $n \equiv 2k+2a+1$ or $6k-2a-1 \pmod{8k}$}, \\
0, &\text{otherwise}.
\end{cases}
\end{equation}
Then 
\begin{equation} \label{family3}
\begin{aligned}
\sum_{n_1,\dots,n_k \geq 0}& (q^2;q^2)_{n_k}\frac{q^{2n_1^2+ \cdots +2n_{k-1}^2 + 2n_{a+1} + \cdots + 2n_{k-1}}}{(-q;q^2)_{n_1+\delta_{a,0}}} \prod_{i=1}^{k-1} \begin{bmatrix} n_{i+1}  + \delta_{i,a}\\ n_i \end{bmatrix}_{q^2} \\
&``=" -\frac{1}{2}\sum_{n \geq 0} n\chi_{8k}^{(a)}(n)q^{\frac{n^2 - (2k-2a-1)^2}{8k}}.
\end{aligned}
\end{equation} \\
\item[($iii$)] Let $\chi_{4k-2}^{(a)}(n)$ be the even periodic function modulo $4k-2$ defined by
\begin{equation} \label{char3}
\chi_{4k-2}^{(a)}(n) = 
\begin{cases}
1, &\text{if $n \equiv  \pm (2k-2a-1) \pmod{4k-2}$}, \\
0, &\text{otherwise}.
\end{cases}
\end{equation}
Then
\begin{equation} \label{family4}
\begin{aligned}
\sum_{n_1,\dots,n_k \geq 0}& (q)_{n_k}\frac{q^{n_1^2+ \cdots +n_{k-1}^2 + n_{a+1} + \cdots + n_{k-1}}(-1)_{n_1+\delta_{a,0}}}{(q;q^2)_{n_1+\delta_{a,0}}} \prod_{i=1}^{k-1} \begin{bmatrix} n_{i+1} + \delta_{i,a} \\ n_i \end{bmatrix} \\
&``=" - (1+\delta_{a,0})\frac{1}{2}\sum_{n \geq 0} n\chi_{4k-2}^{(a)}(n)q^{\frac{n^2 - (2k-2a-1)^2}{8(2k-1)}}.
\end{aligned}
\end{equation}
\end{itemize}
\end{theorem}

%In general, we say that a $q$-hypergeometric series $f(q)$ satisfies a strange identity if it agrees at roots of unity (or some subset thereof) with a partial theta %function of the form 
%\begin{equation}
%\sum_{n \geq 0} n^{\nu}\chi(n)q^{\frac{n^2-a}{b}}.
%\end{equation}  
%Here $\nu \in \{0,1\}$, $a \geq 0$ and $b >0$ are integers, and $\chi(n)$ is a periodic function such that $\chi(n) \neq 0$ only if $\frac{n^2-a}{b} \in %\mathbb{Z}$.   We write
%\begin{equation}
%f(q) ``=" \sum_{n \geq 0} n^{\nu}\chi(n)q^{\frac{n^2-a}{b}}.
%\end{equation}

%Much of this 

The rest of this paper is organized as follows.    In the next section we introduce the basic notions from the theory of Bailey pairs and prove Zagier's strange identity along with the base cases of the strange identities in Theorem \ref{main}.    To prove the multisum identities requires some further development of the theory of Bailey pairs, which we undertake in Section 3.   Section 4 is then devoted to Hikami's identities.   In Section 5 we complete the proof of Theorem \ref{main}, and  we close in Section 6 with some remarks.

\section{Bailey pairs And Zagier's identity}
%A \emph{Bailey pair} relative to $a$ is a pair of sequences $(\alpha_n,\beta_n)_{n \geq 0}$ satisfying
%\begin{equation} \label{pairdef}
%\beta_n = \sum_{k=0}^n \frac{\alpha_k}{(q)_{n-k}(aq)_{n+k}}, 
%\end{equation} 
%or equivalently 
%\begin{equation} \label{pairdefbis}
%\alpha_n = \frac{(1-aq^{2n})}{1-a} \sum_{j=0}^n \frac{(a)_{n+j}(-1)^{n-j}q^{\binom{n-j}{2}}}{(q)_{n-j}}\beta_j.
%\end{equation}
%Here we have used the standard $q$-hypergeometric notation, 
%\begin{equation*}
%(a)_n = (a;q)_n = \prod_{k=1}^{n} (1-aq^{k-1}),
%\end{equation*}
%valid for $n \in \mathbb{N} \cup \{\infty\}$.   
Recall the definition of a Bailey pair from the introduction.   From now on we shall say ``Bailey pair relative to $a$" instead of ``Bailey pair relative to $(a,q)$" unless the second parameter is something other than $q$.    The \emph{Bailey lemma} says that if $(\alpha_n,\beta_n)$ is a Bailey pair relative to $a$, then so is $(\alpha_n',\beta_n')$, where 

\begin{equation} \label{alphaprimedef}
\alpha'_n = \frac{(\rho_1)_n(\rho_2)_n(aq/\rho_1 \rho_2)^n}{(aq/\rho_1)_n(aq/\rho_2)_n}\alpha_n
\end{equation} 

\noindent and

\begin{equation} \label{betaprimedef}
\beta'_n = \sum_{k=0}^n\frac{(\rho_1)_k(\rho_2)_k(aq/\rho_1 \rho_2)_{n-k} (aq/\rho_1 \rho_2)^k}{(aq/\rho_1)_n(aq/\rho_2)_n(q)_{n-k}} \beta_k.
\end{equation}
A useful limiting form of the Bailey lemma is found by putting \eqref{alphaprimedef} and \eqref{betaprimedef} into \eqref{pairdef} and letting $n \to \infty$, giving

\begin{equation} \label{limitBailey}
\sum_{n \geq 0} (\rho_1)_n(\rho_2)_n (aq/\rho_1 \rho_2)^n \beta_n = \frac{(aq/\rho_1)_{\infty}(aq/\rho_2)_{\infty}}{(aq)_{\infty}(aq/\rho_1 \rho_2)_{\infty}} \sum_{n \geq 0} \frac{(\rho_1)_n(\rho_2)_n(aq/\rho_1 \rho_2)^n }{(aq/\rho_1)_n(aq/\rho_2)_n}\alpha_n,
\end{equation}
%For more on Bailey pairs and the Bailey lemma, see \cite{An1,An2,Wa1}.
provided both sides converge absolutely.   The case $(a,\rho_1,\rho_2) = (x^2q,xq,q)$ of \eqref{limitBailey} is recorded as the following lemma, which will play a key role throughout the paper.

\begin{lemma} \label{x^2qlemma}
If $(\alpha_n,\beta_n)$ is a Bailey pair relative to $x^2q$, then we have
\begin{equation*}
(1-x)\sum_{n \geq 0} (xq)_n(q)_n x^n \beta_n = (1-x^2q) \sum_{n \geq 0} \frac{(q)_n}{(x^2q)_n} x^n \alpha_n.
\end{equation*}
\end{lemma}

We are now ready to prove Zagier's strange identity, along with the base cases of the families of identities in Theorem \ref{main}.   In each case we make use of a specialization of Slater's Bailey pair relative to $a$ \cite{Sl1},
\begin{equation} \label{Slaterpairalpha}
\alpha_n = \frac{(a)_n(1-aq^{2n})(-1)^nq^{\binom{n}{2}}(b)_n(c)_n\left(\frac{aq}{bc}\right)^n}{(q)_n(1-a)(\frac{aq}{b})_n(\frac{aq}{c})_n}
\end{equation}
and
\begin{equation} \label{Slaterpairbeta}
\beta_n = \frac{(\frac{aq}{bc})_n}{(q)_n(\frac{aq}{b})_n(\frac{aq}{c})_n}.
\end{equation}

\begin{proof}[Proof of \eqref{Zagierstrange}]
Take $a=x^2q$ and $b,c \to \infty$ in \eqref{Slaterpairalpha} and \eqref{Slaterpairbeta} to obtain the Bailey pair relative to $x^2q$,
\begin{equation} \label{BPZagieralpha}
\alpha_n = \frac{(x^2q)_n(1-x^2q^{2n+1})(-1)^nx^{2n}q^{n(3n+1)/2}}{(q)_n(1-x^2q)}
\end{equation}
and
\begin{equation} \label{BPZagierbeta}
\beta_n = \frac{1}{(q)_n}.
\end{equation} 
Using this in Lemma \ref{x^2qlemma} we obtain
\begin{equation*}
(1-x)\sum_{n \geq 0} (xq)_n x^n = \sum_{n \geq 0} (-1)^nx^{3n}q^{n(3n+1)/2}(1-x^2q^{2n+1}).
\end{equation*}
Note that for $|q|<1$, the right-hand side converges for all $x \in \mathbb{C}$ while the left-hand side requires $|x| <1$.  To remedy this, we follow Zagier and add and subtract $(x)_{\infty}$ on the left-hand side to obtain
\begin{equation*}
(1-x)\sum_{n \geq 0} \left( (xq)_n - (xq)_{\infty}\right)x^n + (xq)_{\infty} = \sum_{n \geq 0} (-1)^nx^{3n}q^{n(3n+1)/2}(1-x^2q^{2n+1}).
\end{equation*}
Now replace $x$ by $x^2$ on both sides, multiply by $x$, and take $\frac{d}{dx} | _{x=1}$.     The result is the ``sum of tails" identity,
\begin{equation*}
2\sum_{n \geq 0} \left((q)_n - (q)_{\infty} \right) + (q)_{\infty}\left(-1+ 2\sum_{n \geq 1} \frac{q^n}{1-q^n}\right) = -\sum_{n \geq 1} n \left(\frac{12}{n}\right)q^{(n^2-1)/24}.
\end{equation*}  
%where 
%\begin{equation}
%\chi(n) =
%\begin{cases}
%1, & \text{if $n \equiv \pm 1 \pmod{12}$}, \\
%-1,& \text{if $n \equiv \pm 5 \pmod{12}$}, \\
%0, & \text{otherwise}.
%\end{cases}
%\end{equation}
Letting $q$ approach a root of unity, we obtain Zagier's strange identity, since $(q)_{\infty}$ vanishes to infinite order.  
\end{proof}
We now proceed to give proofs of the base cases of the strange identities in Theorem \ref{main}.    These all use a specialization of \eqref{Slaterpairalpha} and \eqref{Slaterpairbeta} in Lemma \ref{x^2qlemma} and a short computation similar to the one above.

\begin{proof}[Proof of Theorem \ref{main} for $k=1$.]
In the Bailey pair in \eqref{Slaterpairalpha} and \eqref{Slaterpairbeta}, take $a=x^2q$, $b = -xq$, and $c \to \infty$ to obtain the Bailey pair relative to $x^2q$,
\begin{equation} \label{BPfamily1basealpha}
\alpha_n = \frac{(x^2q)_n(1-x^2q^{2n+1})(-1)^nx^nq^{n^2}}{(q)_n(1-x^2q)}
\end{equation}
and
\begin{equation} \label{BPfamily1basebeta}
\beta_n = \frac{1}{(q)_n(-xq)_n}.
\end{equation}
Using this in Lemma \ref{x^2qlemma}, we have
\begin{align*}
(1-x) \sum_{n \geq 0} \frac{(xq)_n}{(-xq)_n}x^n &= \sum_{n \geq 0}(-1)^nx^{2n}q^{n^2}(1-x^2q^{2n+1}) \\
&=  1 + 2\sum_{n \geq 1} \chi_4(n)q^{n^2/{4}}x^n,
\end{align*}
where $\chi_4(n)$ is defined in \eqref{char1}.  Adding and subtracting the product $(x)_{\infty}/(-xq)_{\infty}$ on the left-hand side, we obtain
\begin{equation*}
(1-x)\sum_{n \geq 0} \left( \frac{(xq)_n}{(-xq)_n} - \frac{(xq)_{\infty}}{(-xq)_{\infty}}\right)x^n + \frac{(xq)_{\infty}}{(-xq)_{\infty}} = \sum_{n \geq 0} (-1)^nx^{2n}q^{n^2}(1-x^2q^{2n+1}).
\end{equation*}
%where
%\begin{equation}
%\chi_{4}(n) = 
%\begin{cases}
%1, &\text{if $n \equiv 0 \pmod{4}$}, \\
%-1, &\text{if $n \equiv 2 \pmod{4}$}, \\
%0, &\text{otherwise}.
%\end{cases}
%\end{equation}
Taking $\frac{d}{dx} | _{x=1}$ on both sides gives
\begin{equation*}
-\sum_{n \geq 0} \left(\frac{(q)_n}{(-q)_n} - \frac{(q)_{\infty}}{(-q)_{\infty}} \right)  - 2\frac{(q)_{\infty}}{(-q)_{\infty}}\sum_{n \geq 1} \frac{q^n}{1-q^{2n}} = 2\sum_{n \geq 1} n \chi_4(n)q^{n^2/4}.
\end{equation*} 
Letting $q$ tend to an odd root of unity gives
\begin{equation*}
\sum_{n \geq 0} \frac{(q)_n}{(-q)_n} ``=" -2\sum_{n \geq 1} n \chi_4(n)q^{n^2/4}.
\end{equation*}
This is \eqref{family1} for $k=1$.   

Next, in \eqref{Slaterpairalpha} and \eqref{Slaterpairbeta}, let $q=q^2$, $a=x^2q^2$, $b = -xq$, and $c=-xq^2$.    The resulting Bailey pair relative to $(x^2q^2,q^2)$ is
\begin{equation} \label{BPfamily2basealpha}
\alpha_n = \frac{(x^2q^2;q^2)_n(1-xq^{2n+1})(-1)^nq^{n^2}}{(q^2;q^2)_n(1-x^2q^2)}
\end{equation}  
and
\begin{equation} \label{BPfamily2basebeta}
\beta_n = \frac{(q;q^2)_n}{(q^2;q^2)_n(-xq)_{2n+1}}.
\end{equation}
Using this in Lemma \ref{x^2qlemma} we have
\begin{align*}
(1-x)\sum_{n \geq 0} \frac{(q;q^2)_n(xq^2;q^2)_n}{(-xq)_{2n+1}}x^n &= \sum_{n \geq 0} (-1)^nx^nq^{n^2}(1-xq^{2n+1}) \\
&= 1+ 2\sum_{n \geq 1} \chi_4(n)q^{n^2/4}x^{n/2}.
\end{align*}
Adding and subtracting the product $(q;q^2)_{\infty}(x;q^2)_{\infty}/(-xq)_{\infty}$ on the left-hand side, we have
\begin{align*}
(1-x)\sum_{n \geq 0} &\left(\frac{(q;q^2)_n(xq^2;q^2)_n}{(-xq)_{2n+1}} - \frac{(q;q^2)_{\infty}(xq^2;q^2)_{\infty}}{(-xq)_{\infty}}\right)x^n + \frac{(q;q^2)_{\infty}(xq^2;q^2)_{\infty}}{(-xq)_{\infty}} \\
&=  1+ 2\sum_{n \geq 1} \chi_4(n)q^{n^2/4}x^{n/2}.
\end{align*}
Taking $\frac{d}{dx} | _{x=1}$ on both sides gives
\begin{equation*}
-\sum_{n \geq 0} \left(\frac{(q)_{2n}}{(-q)_{2n+1}} - \frac{(q)_{\infty}}{(-q)_{\infty}}\right) - \frac{(q)_{\infty}}{(-q)_{\infty}}\sum_{n \geq 1} \frac{q^n}{1-q^{2n}} = \sum_{n \geq 0} n\chi_4(n)q^{n^2/4}.
\end{equation*}
Letting $q$ tend to an odd root of unity results in the strange identity
\begin{equation*}
\sum_{n \geq 0} \frac{(q)_{2n}}{(-q)_{2n+1}} ``=" -\sum_{n \geq 0} n\chi_4(n)q^{n^2/4}.
\end{equation*}
This is \eqref{family2} for $k=1$.

Now let $q=q^2$, $a=x^2q^2$, $b=-xq$, and $c \to \infty$ in \eqref{Slaterpairalpha} and \eqref{Slaterpairbeta} to obtain the Bailey pair relative to $(x^2q^2,q^2)$,
\begin{equation*}
\alpha_n = \frac{(x^2q^2;q^2)_n(1-xq^{2n+1})(-1)^nx^nq^{2n^2+n}}{(q^2;q^2)_n(1-x^2q^2)}
\end{equation*}
and
\begin{equation*}
\beta_n = \frac{1}{(q^2;q^2)_n(-xq;q^2)_{n+1}}.
\end{equation*}
Using this in Lemma \ref{x^2qlemma} we have
\begin{align*}
(1-x)\sum_{n \geq 0} \frac{(xq^2;q^2)_n}{(-xq;q^2)_{n+1}}x^n &= \sum_{n \geq 0}(-1)^nx^{2n}q^{2n^2+n}(1-xq^{2n+1}) \\
&= \sum_{n \geq 0} \chi_8^{(0)}(n)q^{(n^2-1)/8}x^{(n-1)/2},
\end{align*}
where $\chi_8^{(0)}(n)$ is defined in \eqref{char2}.   Adding and subtracting the product $(x;q^2)_{\infty}/(-xq;q^2)_{\infty}$ on the left-hand side, we obtain
\begin{equation*}
(1-x)\sum_{n \geq 0} \left(\frac{(xq^2;q^2)_n}{(-xq;q^2)_{n+1}} - \frac{(xq^2;q^2)_{\infty}}{(-xq;q^2)_{\infty}}\right) x^n + \frac{(xq^2;q^2)_{\infty}}{(-xq;q^2)_{\infty}} =  \sum_{n \geq 0} \chi_8^{(0)}(n)q^{(n^2-1)/8}x^{(n-1)/2}.
\end{equation*}
%where
%\begin{equation}
%\chi_{8}(n) = 
%\begin{cases}
%1, &\text{if $n \equiv 1,7 \pmod{8}$}, \\
%-1, &\text{if $n \equiv 3,5 \pmod{8}$}, \\
%0, &\text{otherwise}.
%\end{cases}
%\end{equation}
Taking $\frac{d}{dx} | _{x=1}$ on both sides gives
\begin{align*}
-&\sum_{n \geq 0} \left(\frac{(q^2;q^2)_n}{(-q;q^2)_{n+1}} - \frac{(q^2;q^2)_{\infty}}{(-q;q^2)_{\infty}}\right)   \\
-& \frac{(q^2;q^2)_{\infty}}{(-q;q^2)_{\infty}}\left(-\frac{1}{2} + \sum_{n \geq 1} \frac{q^{2n}}{1-q^{2n}} + \sum_{n \geq 1}\frac{q^{2n-1}}{1+q^{2n-1}}\right) \\
&= \frac{1}{2}\sum_{n \geq 0} n\chi_8^{(0)}(n)q^{(n^2-1)/8}.
\end{align*}
Here we have used the fact that
\begin{align*}
\sum_{n \geq 0} \chi_8^{(0)}(n)q^{(n^2-1)/8} &= \sum_{n \in \mathbb{Z}} (-1)^nq^{2n^2+n} \\
&= \frac{(q^2;q^2)_{\infty}}{(-q;q^2)_{\infty}},
\end{align*}
by Jacobi's triple product identity
\begin{equation} \label{jtp}
\sum_{n \in \mathbb{Z}} z^nq^{n^2} = (-zq;q^2)_{\infty}(-q/z;q^2)_{\infty}(q^2;q^2)_{\infty}.
\end{equation}
Now letting $q$ tend to an $N$th root of unity where $N \not \equiv 2 \pmod{4}$ gives the strange identity
\begin{equation*}
\sum_{n \geq 0} \frac{(q^2;q^2)_n}{(-q;q^2)_{n+1}} ``=" -\frac{1}{2}\sum_{n \geq 0} n\chi_8^{(0)}(n)q^{(n^2-1)/8}.
\end{equation*} 
This is \eqref{family3} for $k=1$.

Finally,  take $a=x^2q$ and $b = -c = xq^{\frac{1}{2}}$ in \eqref{Slaterpairalpha} and \eqref{Slaterpairbeta} to obtain the Bailey pair relative to $x^2q$,
\begin{equation*}
\alpha_n = \frac{(x^2q)_nq^{\binom{n+1}{2}}}{(q)_n(1-x^2q)}
\end{equation*}
and 
\begin{equation*}
\beta_n = \frac{(-q)_n}{(q)_n(x^2q;q^2)_{n+1}}.
\end{equation*}
Using this in Lemma \ref{x^2qlemma} we have the identity
\begin{align*}
(1-x)\sum_{n \geq 0} \frac{(-q)_n(xq)_n}{(x^2q;q^2)_{n+1}}x^n &= \sum_{n \geq 0}x^nq^{\binom{n+1}{2}} \\
&= \sum_{n \geq 0} \chi_2^{(0)}(n)q^{(n^2-1)/8}x^{(n-1)/2},
\end{align*}
%where
%\begin{equation}
%\chi_{2}^{(0)}(n) = 
%\begin{cases}
%1, &\text{if $n \equiv 1 \pmod{2}$}, \\
%0, &\text{otherwise}.
%\end{cases}
%\end{equation}
where $\chi_{2}^{(0)}(n)$ is defined in \eqref{char3}.    Adding and subtracting the infinite product $(-q)_{\infty}(x)_{\infty}/(x^2q;q^2)_{\infty}$ on the left hand side, we obtain
\begin{equation*}
(1-x)\sum_{n \geq 0} \left(\frac{(-q)_n(xq)_n}{(x^2q;q^2)_{n+1}} - \frac{(-q)_{\infty}(xq)_{\infty}}{(x^2q;q^2)_{\infty}}\right)x^n + \frac{(-q)_{\infty}(xq)_{\infty}}{(x^2q;q^2)_{\infty}} =  \sum_{n \geq 0} \chi_2^{(0)}(n)q^{(n^2-1)/8}x^{(n-1)/2}.
\end{equation*}
Taking $\frac{d}{dx} | _{x=1}$ on both sides gives
\begin{equation*}
\begin{aligned}
-\sum_{n \geq 0} \left(\frac{(q^2;q^2)_n}{(q;q^2)_{n+1}} - \frac{(q^2;q^2)_{\infty}}{(q;q^2)_{\infty}}\right) &+ \frac{(q^2;q^2)_{\infty}}{(q;q^2)_{\infty}}\left(\frac{1}{2} - \sum_{n \geq 1}\frac{q^n}{1-q^n} + 2\sum_{n \geq 1} \frac{q^{2n-1}}{1-q^{2n-1}}\right) \\ 
&=  \frac{1}{2}\sum_{n \geq 0} n\chi_2^{(0)}(n)q^{(n^2-1)/8}.
\end{aligned}
\end{equation*}
Here we have used the fact that 
\begin{equation*}
\sum_{n \geq 0} \chi_2^{(0)}(n)q^{(n^2-1)/8} = \frac{(q^2;q^2)_{\infty}}{(q;q^2)_{\infty}},
\end{equation*}
which follows from \eqref{jtp}.
Now letting $q$ tend to an even root of unity gives the strange identity
\begin{equation*}
\sum_{n \geq 0} \frac{(q^2;q^2)_n}{(q;q^2)_{n+1}}``=" -\sum_{n \geq 0} n\chi_2^{(0)}(n)q^{(n^2-1)/8}.
\end{equation*}
This is \eqref{family4} at $k=1$.
\end{proof}
%We close this section with some remarks.   First, the observant reader may have noticed that the case $k=1$ of \eqref{family4} is just the case $k=1$ of %\eqref{family3} with $q=-q$.    However, \eqref{family4} comes from a different Bailey pair and this relation does not persist for $k \geq 2$, so we gave a proof %in anticipation of the general argument in Section 5.     Second, some of the ``sums of tails" identities in this section appear in \cite{An-Ji-On1} and %\cite{Foetal1}, where the corresponding strange identities are implicit.    Third, note that ...  implies the quantum $q$-series identity
%\begin{equation*}
%\sum_{n \geq 0} \frac{(q)_n}{(-q)_n} =_q 2\sum_{n \geq 0} \frac{(q)_{2n}}{(-q)_{2n+1}},
%\end{equation*}
%in the sense of \cite{Lo.5}.   In other words,the two sides are equal at any odd root of unity $q$.    Finally, we leave it to the interested reader to find other nice   
We leave it to the interested reader to find more single-sum strange identities arising from Bailey pairs in Lemma \ref{x^2qlemma}.    This lemma is admittedly somewhat restrictive, due the presence of the term $(q)_n/(x^2q)_n$ in the summand of the right-hand side, but there must surely be other nice examples.  For now we turn to using multisum Bailey pairs in Lemma \ref{x^2qlemma} in order to prove Hikami's identity \eqref{Hikamistrange} and Theorem \ref{main}. 

\section{More on Bailey pairs}

Our goal in this section is to prove Lemma \ref{aux3} below.    We accomplish this via a sequence of auxiliary results.   We first record a lemma from \cite{Lo1}.   

\begin{lemma}{\cite[Lemma 3.2]{Lo1}} \label{recentlemma}
If $(\alpha_n,\beta_n)$ is a Bailey pair relative to $a$, then
so is $(\gamma^*_n - \gamma^*_{n-1},\beta^*_n)$, where
$\gamma^*_{-1} = 0$,
\begin{equation} \label{gammastarbis}
\gamma^*_n = \frac{(aq/b)_n(-b)^nq^{n(n+1)/2}}{(bq)_n}\sum_{r =
0}^n \frac{(b)_r(-b)^{-r}q^{-r(r-1)/2}\alpha_r}{(aq/b)_r},
\end{equation}
and
\begin{equation} \label{betastarbis}
\beta^*_n = \frac{(b)_nq^n}{(bq)_n}\beta_n.
\end{equation}
\end{lemma}

This may be used to deduce the following. 
\begin{lemma} \label{aux1}
If $(\alpha_n,\beta_n)$ is a Bailey pair relative to $a$, then so is $(\alpha_n',\beta_n')$, where
\begin{equation} \label{aux1beta}
\beta_n' = (1-q^n)\beta_n
\end{equation}
and
\begin{equation} \label{aux1alpha}
\alpha_n' = (1-q^n)\alpha_n + a^{n-1}q^{n^2-n}(1-aq^{2n}) \sum_{r=0}^{n-1} a^{-r}q^{-r^2} \alpha_r.
\end{equation}
\end{lemma}

\begin{proof}
Suppose that $(\alpha_n,\beta_n)$ is a Bailey pair relative to $a$.   Using the linearity of Bailey pairs along with Lemma \ref{recentlemma},  we have that 
\begin{equation*}
(\alpha_n',\beta_n') = \left( \alpha_n - (\gamma_n^* - \gamma_{n-1}^*),\beta_n - \beta_n^* \right)
\end{equation*}
is a Bailey pair relative to $a$.   Using the case $b \to 0$ of equations \eqref{gammastarbis} and \eqref{betastarbis}, we have that 
\begin{align*}
\alpha_n' &= \alpha_n - (\gamma_n^* - \gamma_{n-1}^*) \\
&= \alpha_n - \left(a^nq^{n^2+n}\sum_{r=0}^na^{-r}q^{-r^2}\alpha_r - a^{n-1}q^{n^2-n}\sum_{r=0}^{n-1}a^{-r}q^{-r^2}\alpha_r\right) \\
&= \alpha_n - \left(-a^{n-1}q^{n^2-n}(1-aq^{2n})\sum_{r=0}^{n-1}a^{-r}q^{-r^2}\alpha_r + q^n\alpha_n\right) \\
&= (1-q^n)\alpha_n + a^{n-1}q^{n^2-n}(1-aq^{2n}) \sum_{r=0}^{n-1} a^{-r}q^{-r^2} \alpha_r
\end{align*}
and 
\begin{equation*}
\beta_n' = (1-q^n)\beta_n.
\end{equation*}
This completes the proof.
\end{proof}

The next lemma extends Theorem 1.2 of \cite{Lo-Os1} from the case $a=1$ to arbitrary $a$.   
\begin{lemma} \label{aux2}
If $(\alpha_n,\beta_n)$ is a Bailey pair relative to $a$ with $\alpha_0 = \beta_0 = 0$, then $(\alpha_n',\beta_n')$ is a Bailey pair relative to $aq$, where
\begin{equation} \label{aux2beta}
\beta_n' = \beta_{n+1}
\end{equation}
and
\begin{equation} \label{aux2alpha}
\alpha_n' = \frac{1}{1-aq}\left( \frac{1}{1-aq^{2n+2}} \alpha_{n+1} - \frac{aq^{2n}}{1-aq^{2n}}\alpha_n\right).
\end{equation}
\end{lemma}

\begin{proof}
We first recall the Bailey pair inversion, which says that \eqref{pairdef} holds if and only if
\begin{equation} \label{pairdefbis}
\alpha_n = \frac{(1-aq^{2n})}{1-a} \sum_{j=0}^n \frac{(a)_{n+j}(-1)^{n-j}q^{\binom{n-j}{2}}}{(q)_{n-j}}\beta_j.
\end{equation}
Now suppose that $(\alpha_n,\beta_n)$ is a Bailey pair relative to $a$ with $\alpha_0 = \beta_0 = 0$.    We set $\beta_j' = \beta_{j+1}$ and use \eqref{pairdefbis} to calculate $\alpha_n'$ with $a=aq$ as follows:
\begin{align*}
\alpha_n' &= \frac{1-aq^{2n+1}}{1-aq}\sum_{j=0}^n \frac{(aq)_{n+j}(-1)^{n-j}q^{\binom{n-j}{2}}}{(q)_{n-j}}\beta_j' \\
&= \frac{1}{1-aq} \sum_{j=0}^n \frac{(aq)_{n+j}(-1)^{n-j}q^{\binom{n-j}{2}}}{(q)_{n-j}}\beta_{j+1}(1-aq^{n+j+1}+aq^{n+j+1}(1-q^{n-j})) \\
&= \frac{1}{1-aq} \Bigg( \sum_{j=0}^n \frac{(aq)_{n+j+1}(-1)^{n-j}q^{\binom{n-j}{2}}}{(q)_{n-j}}\beta_{j+1} + a\sum_{j=0}^{n-1} \frac{(aq)_{n+j}(-1)^{n-j}q^{\binom{n-j}{2}+n+j+1}}{(q)_{n-j-1}}\beta_{j+1} \Bigg) \\
&= \frac{1}{1-aq} \Bigg( \sum_{j=1}^{n+1} \frac{(aq)_{n+j}(-1)^{n+1-j}q^{\binom{n+1-j}{2}}}{(q)_{n+1-j}}\beta_{j} + a\sum_{j=1}^{n} \frac{(aq)_{n+j-1}(-1)^{n-j+1}q^{\binom{n-j+1}{2}+n+j}}{(q)_{n-j}}\beta_{j} \Bigg) \\
&= \frac{1}{1-aq} \Bigg( \sum_{j=1}^{n+1} \frac{(a)_{n+j+1}(-1)^{n+1-j}q^{\binom{n+1-j}{2}}}{(1-a)(q)_{n+1-j}}\beta_{j} - aq^{2n} \sum_{j=1}^{n} \frac{(a)_{n+j}(-1)^{n-j}q^{\binom{n-j}{2}}}{(1-a)(q)_{n-j}}\beta_{j} \Bigg) \\
&=\frac{1}{1-aq}\Bigg( \frac{1}{1-aq^{2n+2}}\alpha_{n+1} - \frac{aq^{2n}}{1-aq^{2n}}\alpha_n\Bigg),
\end{align*}
as desired.
\end{proof}

Finally we have our key lemma.  
\begin{lemma} \label{aux3}
 If $(\alpha_n,\beta_n)$ is a Bailey pair relative to $a$, then $(\alpha_n'',\beta_n'')$ is a Bailey pair relative to $aq$, where
 \begin{equation} \label{aux3beta}
 \beta_n'' = (1-q^{n+1})\beta_{n+1}
 \end{equation}
 and
 \begin{equation} \label{aux3alpha}
 \alpha_n'' = \frac{1}{1-aq}\left( \frac{1-q^{n+1}}{1-aq^{2n+2}} \alpha_{n+1} + \frac{q^{n}(1-aq^n)}{1-aq^{2n}}\alpha_n\right).
 \end{equation}
\end{lemma}

\begin{proof}
Suppose that $(\alpha_n,\beta_n)$ is a Bailey pair relative to $a$.     We apply Lemma \ref{aux2} to the Bailey pair $(\alpha_n',\beta_n')$ relative to $a$ resulting from an application of Lemma \ref{aux1}.    Equation \eqref{aux3beta} follows directly from \eqref{aux1beta} and \eqref{aux2beta}.    For \eqref{aux3alpha} we use \eqref{aux1alpha} and \eqref{aux2alpha} to compute
\begin{align*}
\alpha_n'' = \frac{1}{1-aq}&\Bigg(\frac{1}{1-aq^{2n+2}}\alpha_{n+1}' - \frac{aq^{2n}}{1-aq^{2n}}\alpha_n'\Bigg) \\
=\frac{1}{1-aq}&\Bigg(\frac{1-q^{n+1}}{1-aq^{2n+2}}\alpha_{n+1} + \frac{1}{1-aq^{2n+2}}a^nq^{n^2+n}(1-aq^{2n+2})\sum_{r=0}^na^{-r}q^{-r^2}\alpha_r  \\
&- \frac{aq^{2n}}{1-aq^{2n}}(1-q^n)\alpha_n - \frac{a^nq^{n^2+n}}{1-aq^{2n}}(1-aq^{2n}) \sum_{r=0}^{n-1}a^{-r}q^{-r^2}\alpha_r \Bigg) \\
= \frac{1}{1-aq}&\Bigg(\frac{1-q^{n+1}}{1-aq^{2n+2}}\alpha_{n+1} - \frac{aq^{2n}}{1-aq^{2n}}(1-q^n)\alpha_n + q^n\alpha_n\Bigg) \\
= \frac{1}{1-aq}&\Bigg(\frac{1-q^{n+1}}{1-aq^{2n+2}}\alpha_{n+1} + \frac{q^{n}(1-aq^n)}{1-aq^{2n}}\alpha_n \Bigg),
\end{align*}
as desired.
\end{proof}

%Next we record the case $(a,\rho_1,\rho_2) = (x^2q,xq,q)$ of \eqref{limitBailey}.

%\begin{lemma} \label{x^2qlemma}
%If $(\alpha_n,\beta_n)$ is a Bailey pair relative to $x^2q$, then subject to convergence conditions we have
%\begin{equation}
%(1-x)\sum_{n \geq 0} (xq)_n(q)_n x^n \beta_n = (1-x^2q) \sum_{n \geq 0} \frac{(q)_n}{(x^2q)_n} x^n \alpha_n.
%\end{equation}
%\end{lemma}

\section{Hikami's identities}
To obtain Hikami's identities we need a more general Bailey pair than the one in \eqref{BPZagieralpha} and \eqref{BPZagierbeta}.    
\begin{proposition} \label{Hikamipair}
For $k \geq 1$ and $0 \leq a \leq k-1$, the following is a Bailey pair relative to $x^2q$:
\begin{equation} \label{Hikamialpha}
\alpha_n = \frac{(x^2q)_n}{(q)_n(1-x^2q)}(-1)^nx^{2kn}q^{\binom{n+1}{2} + (a+1)n^2 + (k-a-1)(n^2+n)}(1-x^{2(a+1)}q^{(a+1)(2n+1)})
\end{equation}
and
\begin{equation} \label{Hikamibeta}
\beta_n = \beta_{n_k} = \sum_{n_1,n_2,\dots,n_{k-1} \geq 0} \frac{q^{n_1^2+ \cdots +n_{k-1}^2 + n_{a+1} + \cdots + n_{k-1}}x^{2n_1 + \cdots + 2n_{k-1}}}{(q)_{n_k}} \prod_{i=1}^{k-1} \begin{bmatrix} n_{i+1} + \delta_{a,i} \\ n_i \end{bmatrix}.
\end{equation}
\end{proposition}
  
\begin{proof}
%By \eqref{Slaterpairalpha} and \eqref{Slaterpairbeta} with $b,c \to \infty$ and $a=x^2$, we have that $(\alpha_n,\beta_n)$ is a Bailey pair relative to $x^2$, %where
%\begin{equation}
%\alpha_n = \frac{(x^2)_n(1-x^2q^{2n})}{(q)_n(1-x^2)}(-1)^nq^{\binom{n}{2}+n^2}x^{2n}
%\end{equation}
%and
%\begin{equation}
%\beta_n = \frac{1}{(q)_n}.
%\end{equation}
%Start with the Bailey pair in \eqref{BPZagieralpha} and \eqref{BPZagierbeta}.   
We start by settting $a=x^2$ and letting $b,c \to \infty$ in \eqref{Slaterpairalpha} and \eqref{Slaterpairbeta} to obtain a Bailey pair relative to $x^2$, 
\begin{equation*}
\alpha_n = \frac{(x^2)_n(1-x^2q^{2n})(-1)^nx^{2n}q^{n(3n-1)/2}}{(q)_n(1-x^2)}
\end{equation*}
and
\begin{equation*}
\beta_n = \frac{1}{(q)_n}.
\end{equation*}
Assuming for a moment that $a \geq 1$, we insert this Bailey pair in \eqref{alphaprimedef} and \eqref{betaprimedef} and iterate $a$ times with $\rho_1,\rho_2 \to \infty$ to obtain another Bailey pair relative to $x^2$, 
\begin{equation*}
\alpha_n = \frac{(x^2)_n(1-x^2q^{2n})}{(q)_n(1-x^2)}(-1)^nq^{\binom{n}{2}+(a+1)n^2}x^{2(a+1)n}
\end{equation*}
and
\begin{equation*}
\beta_n = \beta_{n_{a+1}} = \sum_{n_1,\dots, n_{a}  \geq 0} \frac{q^{n_1^2+\cdots +n_{a}^2}x^{2n_1 + \cdots + 2n_a}}{(q)_{n_{a+1} - n_a} \cdots (q)_{n_2-n_1}(q)_{n_1}}.
\end{equation*}
Next we apply Lemma \ref{aux3} to the above to obtain a Bailey pair relative to $x^2q$, 
\begin{align*}
\alpha_n = \frac{1}{1-x^2q}&\Bigg(\frac{(1-q^{n+1})}{(1-x^2q^{2n+2})}\frac{(x^2)_{n+1}(1-x^2q^{2n+2})}{(q)_{n+1}(1-x^2)}(-1)^{n+1}q^{\binom{n+1}{2}+ (a+1)(n+1)^2}x^{(a+1)(2n+2)} \\
&+ \frac{q^n(1-x^2q^n)}{(1-x^2q^{2n})} \frac{(x^2)_n(1-x^2q^{2n})}{(q)_n(1-x^2)}(-1)^nq^{\binom{n}{2} + (a+1)n^2}x^{2(a+1)n}\Bigg) \\
= \frac{1}{1-x^2q} &\Bigg( \frac{(x^2q)_n}{(q)_n}(-1)^{n+1}q^{\binom{n+1}{2} + (a+1)(n+1)^2}x^{(a+1)(2n+2)} \\
&+ \frac{(x^2q)_n}{(q)_n}(-1)^{n+1}q^{\binom{n+1}{2} + (a+1)n^2}x^{(a+1)(2n)}\Bigg) \\
=\frac{1}{1-x^2q}&\frac{(x^2q)_n}{(q)_n}(-1)^nq^{\binom{n+1}{2} + (a+1)n^2}x^{2(a+1)n}(1-x^{2(a+1)}q^{(a+1)(2n+1)})
\end{align*} 
and 
\begin{equation} \label{Hikamibetaafteraiterations}
\beta_n = \beta_{n_{a+1}}  = \sum_{n_1,\dots, n_{a}  \geq 0} \frac{q^{n_1^2+\cdots +n_{a}^2}x^{2n_1 + \cdots + 2n_a}(1-q^{n_{a+1}+1})}{(q)_{n_{a+1}+1 - n_a} \cdots (q)_{n_2-n_1}(q)_{n_1}}.
\end{equation}
Finally, we iterate this $k-1-a$ times using \eqref{alphaprimedef} and \eqref{betaprimedef} with $\rho_1,\rho_2 \to \infty$ to obtain the Bailey pair relative to $x^2q$,
\begin{equation*}
\alpha_n = \frac{(x^2q)_n}{(q)_n(1-x^2q)}(-1)^nx^{2kn}q^{\binom{n+1}{2} + (a+1)n^2 + (k-a-1)(n^2+n)}(1-x^{2(a+1)}q^{(a+1)(2n+1)}).
\end{equation*}
and
\begin{equation*}
\beta_n = \beta_{n_k}  = \sum_{n_1,\dots, n_{k-1}  \geq 0} \frac{q^{n_1^2+\cdots +n_{k-1}^2 + n_{a+1} + \cdots + n_{k-1}}x^{2n_1 + \cdots + 2n_{k-1}}(1-q^{n_{a+1}+1})}{(q)_{n_k - n_{k-1}} \cdots (q)_{n_{a+2} - n_{a+1}}(q)_{n_{a+1}+1 - n_a}(q)_{n_a - n_{a-1}} \cdots (q)_{n_2-n_1}(q)_{n_1}}.
\end{equation*}
In this last expression, multiplying the numerator and denominator by
\begin{equation*}
(q)_{n_k} \cdots (q)_{n_2}
\end{equation*} 
and using the definition of the $q$-binomial coefficient \eqref{qbincoeff} gives the expression in the statement of the proposition.   For $a=0$, instead of a multisum at \eqref{Hikamibetaafteraiterations} we have
\begin{equation*}
\beta_n = \beta_{n_1} = \frac{1}{(q)_{n_1}}.
\end{equation*}  
The rest of the argument is the same.   This completes the proof.
\end{proof}

We note one corollary for later use.    This is Hikami's ``Andrews-Gordon variant" and may be compared with \eqref{AG}.
\begin{corollary} \label{AGvariant}
Recall the periodic function $\chi_{8m+4}^{(a)}(n)$ defined in \eqref{Hikamichar}.   For $0 \leq a \leq k-1$ we have
\begin{align*} 
\sum_{n_1,n_2,\dots,n_{k-1} \geq 0}& \frac{q^{n_1^2+ \cdots +n_{k-1}^2 + n_{a+1} + \cdots + n_{k-1}}}{(q)_{n_{k-1}}} \prod_{i=1}^{k-2} \begin{bmatrix} n_{i+1} + \delta_{a,i} \\ n_i \end{bmatrix} \\
&= \frac{1}{(q)_{\infty}} \sum_{n \geq 0} \chi_{8m+4}^{(a)}(n)q^{\frac{n^2 - (2k-2a-1)^2}{8(2k+1)}} \\
&= \prod_{n \not \equiv 0, \pm (a+1) \pmod{2k+1}} \frac{1}{1-q^n}.
\end{align*}
\end{corollary}

\begin{proof}
%First, use the Bailey pair in \eqref{Hikamibeta} and \eqref{Hikamialpha} with $x=1$ in \eqref{limitBailey} with
%$\rho_1, \rho_2 \to \infty$ to obtain
We use the Bailey pair in \eqref{Hikamibeta} and \eqref{Hikamialpha} with $x=1$ in the definition of a Bailey pair \eqref{pairdef} and let $n \to \infty$.    This gives
\begin{align*}
\sum_{n_1,n_2,\dots,n_{k-1} \geq 0}& \frac{q^{n_1^2+ \cdots + n_{k-1}^2 + n_{a+1} + \cdots +n_{k-1}}}{(q)_{n_{k-1}}} \prod_{i=1}^{k-2} \begin{bmatrix} n_{i+1} + \delta_{a,i} \\ n_i \end{bmatrix} \\
&= \frac{1}{(q)_{\infty}}\sum_{n \geq 0} (-1)^nq^{(k-a-1)(n^2+n)+\binom{n+1}{2} + (a+1)n^2}(1-q^{(a+1)(2n+1)}) \\
&= \frac{1}{(q)_{\infty}}\sum_{n \geq 0} \chi_{8m+4}^{(a)}(n)q^{\frac{n^2 - (2k-2a-1)^2}{8(2k+1)}}  \\
&= \frac{1}{(q)_{\infty}}\sum_{n \in \mathbb{Z}} (-1)^nq^{\frac{2k+1}{2}n^2 + \frac{2k-2a-1}{2}n} \\
&= \frac{(q^{a+1},q^{2k-a},q^{2k+1};q^{2k+1})_{\infty}}{(q)_{\infty}},
\end{align*}   
by Jacobi's triple product identity \eqref{jtp}.
%\begin{equation}
%\sum_{n \in \mathbb{Z}} z^nq^{n^2} = (-q/z,-zq,q^2;q^2)_{\infty}.
%\end{equation}
\end{proof}
    
%The term $\begin{bmatrix} n_{a+1} + 1 \\ n_a \end{bmatrix}$ is understood to be $1$ if $a =0$ or $k-1$.

We are now almost ready to prove Hikami's identities.    We need one last $q$-series lemma.
\begin{lemma} \label{qlemma}
We have
\begin{equation} \label{qlemmaeq1}
\sum_{n \geq 0} x^n \begin{bmatrix} n \\ k \end{bmatrix} = \frac{x^k}{(x)_{k+1}}
\end{equation}
and
\begin{equation} \label{qlemmaeq2}
\sum_{n \geq 0} x^n \begin{bmatrix} n + 1 \\ k \end{bmatrix} = \frac{x^{k - \chi(k \neq 0)}}{(x)_{k+1}}
\end{equation}
\end{lemma}

\begin{proof}
We have
\begin{align*}
\sum_{n \geq 0} x^n \begin{bmatrix} n \\ k \end{bmatrix} &= \sum_{n \geq k} x^n \begin{bmatrix} n \\ k \end{bmatrix} \\
&= \sum_{n \geq 0} x^{n+k} \begin{bmatrix} n + k \\ k \end{bmatrix} \\
&= \frac{x^k}{(x)_{k+1}},
\end{align*}
by \cite[Eq. (3.3.7)]{An.5}.   This is \eqref{qlemmaeq1}.

For \eqref{qlemmaeq2}, the case $k=0$ is clear.   If $k \geq 1$, then
\begin{align*}
\sum_{n \geq 0} x^n \begin{bmatrix} n+1 \\ k \end{bmatrix} &= x^{-1}\sum_{n \geq 0} x^{n+1} \begin{bmatrix} n+1 \\ k \end{bmatrix} \\
&= x^{-1}\sum_{n \geq 1} x^n \begin{bmatrix} n  \\ k \end{bmatrix} \\
&= x^{-1}\sum_{n \geq 0} x^n \begin{bmatrix} n  \\ k \end{bmatrix},
\end{align*}
and the result follows from \eqref{qlemmaeq1}.
\end{proof}

\begin{proof}[Proof of \eqref{Hikamistrange}]
We begin by inserting the Bailey pair in \eqref{Hikamibeta} and \eqref{Hikamialpha} into Lemma \ref{x^2qlemma}.    This gives
\begin{align*}
(1-x)\sum_{n_1,\dots,n_k \geq 0}& (xq)_{n_k}q^{n_1^2+ \cdots +n_{k-1}^2 + n_{a+1} + \cdots + n_{k-1}}x^{2n_1 + \cdots + 2n_{k-1} + n_k} \prod_{i=1}^{k-1} \begin{bmatrix} n_{i+1} + \delta_{a,i} \\ n_i \end{bmatrix}  \\
&= \sum_{n \geq 0} (-1)^nx^{(2k+1)n}q^{\binom{n+1}{2} + (a+1)n^2+ (k-a-1)(n^2+n)}(1-x^{2(a+1)}q^{(a+1)(2n+1)}).
\end{align*}
A short calculation shows that the right-hand side can be written
\begin{equation*}
\sum_{n \geq 0} \chi_{8k+4}^{(a)}(n)q^{\frac{n^2 - (2k-2a-1)^2}{8(2k+1)}}x^{\frac{n - (2k-2a-1)}{2}},
\end{equation*} 
where $\chi_{8k+4}^{(a)}(n)$ is defined in \eqref{Hikamichar}.  
%the even periodic function modulo $8k+4$ defined by
%\begin{equation}
%\chi_{8k+4}^{(a)}(n) = 
%\begin{cases}
%1, &\text{if $n \equiv 2k-2a-1$ or $6k+2a+5 \pmod{8k+4}$}, \\
%-1, &\text{if $n \equiv 2k+2a+3$ or $6k-2a+1 \pmod{8k+4}$}, \\
%0, &\text{otherwise}.
%\end{cases}
%\end{equation}
As for the left-hand side, we add and subtract the product $(x)_{\infty}$ to obtain 
\begin{align*}
(1-x)\sum_{n_1,\dots,n_k \geq 0}& ((xq)_{n_k} - (xq)_{\infty})q^{n_1^2+ \cdots +n_{k-1}^2 + n_{a+1} + \cdots + n_{k-1}}x^{2n_1 + \cdots + 2n_{k-1} + n_k} \prod_{i=1}^{k-1} \begin{bmatrix} n_{i+1} + \delta_{a,i} \\ n_i \end{bmatrix}  \\
+ (x)_{\infty} \sum_{n_1,\dots,n_k \geq 0}&q^{n_1^2+ \cdots +n_{k-1}^2 + n_{a+1} + \cdots + n_{k-1}}x^{2n_1 + \cdots + 2n_{k-1} + n_k} \prod_{i=1}^{k-1} \begin{bmatrix} n_{i+1} + \delta_{a,i} \\ n_i \end{bmatrix},
\end{align*}
and then using Lemma \eqref{qlemma} to eliminate the $n_k$ variable in the second term gives 
\begin{equation*}
\begin{aligned}
(xq)_{\infty} \sum_{n_1,\dots,n_{k-1} \geq 0} &\frac{q^{n_1^2+ \cdots +n_{k-1}^2 + n_{a+1} + \cdots + n_{k-1}}x^{2n_1 + \cdots + 2n_{k-1} + n_{k-1} - \delta_{a,k-1}\chi(n_{k-1} > 0)}}{(xq)_{n_{k-1}}} \\
&\times \prod_{i=1}^{k-2} \begin{bmatrix} n_{i+1} + \delta_{a,i} \\ n_i \end{bmatrix} \\
+(1-x)\sum_{n_1,\dots,n_k \geq 0}& ((xq)_{n_k} - (xq)_{\infty})q^{n_1^2+n_2^2+ \cdots +n_{k-1}^2 + n_{a+1} + \cdots + n_{k-1}}x^{2n_1 + \cdots + 2n_{k-1} + n_k} \\
&\times \prod_{i=1}^{k-1} \begin{bmatrix} n_{i+1} + \delta_{a,i} \\ n_i \end{bmatrix} \\
= \sum_{n \geq 0}& \chi_{8m+4}^{(a)}(n)q^{\frac{n^2 - (2k-2a-1)^2}{8(2k+1)}}x^{\frac{n - (2k-2a-1)}{2}}.
\end{aligned}
\end{equation*}
This is Lemma 10 in Hikami's paper \cite{Hi1}, with a slight correction in the case $a=k-1$.   Follwing Hikami, we differentiate with respect to $x$, set $x=1$, and use Corollary \ref{AGvariant} to obtain
\begin{equation*}
\begin{aligned}
(q)_{\infty}\sum_{n_1,\dots,n_{k-1} \geq 0}& \frac{q^{n_1^2+ \cdots +n_{k-1}^2 + n_{a+1} + \cdots + n_{k-1}}}{(q)_{n_{k-1}}} \prod_{i=1}^{k-2} \begin{bmatrix} n_{i+1} + \delta_{a,i} \\ n_i \end{bmatrix}\\
&\times \left(2n_1 +\cdots + 2n_{k-1} + n_{k-1} - \delta_{a,k-1}\chi(n_{k-1} > 0)+ \sum_{j=1}^{n_{k-1}} \frac{q^j}{1-q^j}\right) \\
+ (q^{a+1};q^{2k+1})_{\infty}&(q^{2k-a};q^{2k+1})_{\infty}(q^{2k+1};q^{2k+1})_{\infty}\left( \frac{2k-2a-1}{2} - \sum_{j \geq 1} \frac{q^j}{1-q^j}\right) \\
-\sum_{n_1,\dots,n_k \geq 0}& ((q)_{n_k} - (q)_{\infty})q^{n_1^2+ \cdots +n_{k-1}^2 + n_{a+1} + \cdots + n_{k-1}}  \prod_{i=1}^{k-1} \begin{bmatrix} n_{i+1} + \delta_{a,i} \\ n_i \end{bmatrix} \\
&= \frac{1}{2} \sum_{n \geq 0} n\chi_{8k+4}^{(a)}(n)q^{\frac{n^2 - (2k-2a-1)^2}{8(2k+1)}}.
\end{aligned}
\end{equation*}
This might be described as ``a bit of a mess," but the point is that as $q$ tends to a root of unity the first two terms are annihilated, leaving the ``strange" identity
\begin{equation*}
\begin{aligned}
\sum_{n_1,\dots,n_k \geq 0}& (q)_{n_k}q^{n_1^2+ \cdots +n_{k-1}^2 + n_{a+1} + \cdots + n_{k-1}} \prod_{i=1}^{k-1} \begin{bmatrix} n_{i+1} +\delta_{a,i} \\ n_i \end{bmatrix}  \\
&``=" -\frac{1}{2} \sum_{n \geq 0} n\chi_{8k+4}^{(a)}(n)q^{\frac{n^2 - (2k-2a-1)^2}{8(2k+1)}}.
\end{aligned}
\end{equation*}
This is \eqref{Hikamistrange}.
\end{proof}

\section{Proof of Theorem \ref{main}}
In this section we prove the four families of strange identities in Theorem \ref{main}.    In each case we follow the same basic outline as in the previous section.   First, we give a Bailey pair and corresponding Rogers-Ramanujan type identities.   Then we apply Lemma \ref{x^2qlemma} and after some manipulation differentiate to find the strange identity.    Since the results for $k=1$ were established in Section 2, we assume that $k \geq 2$ throughout. 

\subsection{Proof of \eqref{family1}} 
We begin with a Bailey pair.
\begin{proposition}
For $k \geq 2$, the sequences $(\alpha_n,\beta_n)$ form a Bailey pair relative to $x^2q$, where
\begin{equation} \label{family1alpha}
\alpha_n = \frac{(x^2q)_n(1-x^2q^{2n+1})}{(q)_n(1-x^2q)}(-1)^nx^{(2k-1)n}q^{kn^2 + (k-1)n}
\end{equation}
and
\begin{equation} \label{family1beta}
\beta_n = \beta_{n_k} = \sum_{n_1,\dots ,n_{k-1} \geq 0} \frac{q^{n_1^2+ \cdots + n_{k-1}^2 + n_1 + \cdots + n_{k-1}}x^{2n_1 + \cdots + 2n_{k-1}}}{(q)_{n_k}(-xq)_{n_1}}\prod_{i=1}^{k-1} \begin{bmatrix} n_{i+1} \\ n_i \end{bmatrix}.
\end{equation}
\end{proposition}

\begin{proof}
Start with the Bailey pair in \eqref{BPfamily1basealpha} and \eqref{BPfamily1basebeta} and iterate $k-1$ times using \eqref{alphaprimedef} and \eqref{betaprimedef} with $\rho_1, \rho_2 \to \infty$. 
\end{proof}

\begin{corollary} \label{family1RR}
Recall the periodic function $\chi_{4k}(n)$ defined in \eqref{char1}.   For $k \geq 2$ we have
\begin{align*}
\sum_{n_1,\dots ,n_{k-1} \geq 0}& \frac{q^{n_1^2+ \cdots + n_{k-1}^2 + n_1 + \cdots + n_{k-1}}}{(q)_{n_{k-1}}(-xq)_{n_1}}\prod_{i=1}^{k-2} \begin{bmatrix} n_{i+1} \\ n_i \end{bmatrix} \\
&= \frac{1}{(q)_{\infty}} \sum_{n \geq 0} \chi_{4k}(n)q^{\frac{n^2-(k-1)^2}{4k}} \\
&= \prod_{n \neq 0,\pm 1 \pmod{4k}} \frac{1}{1-q^n}.
\end{align*}
\end{corollary}

\begin{proof}
We use the Bailey pair in \eqref{family1alpha} and \eqref{family1beta} with $x=1$ in the definition of a Bailey pair \eqref{pairdef} and let $n \to \infty$.    This gives
\begin{align*}
\sum_{n_1,\dots ,n_{k-1} \geq 0}& \frac{q^{n_1^2+ \cdots + n_{k-1}^2 + n_1 + \cdots + n_{k-1}}}{(q)_{n_{k-1}}(-xq)_{n_1}}\prod_{i=1}^{k-2} \begin{bmatrix} n_{i+1} \\ n_i \end{bmatrix} \\
&= \frac{1}{(q)_{\infty}} \sum_{n \geq 0} (-1)^nq^{kn^2+(k-1)n}(1-q^{2n+1}) \\
&= \frac{1}{(q)_{\infty}} \sum_{n \geq 0} \chi_{4k}(n)q^{\frac{n^2-(k-1)^2}{4k}} \\
&= \frac{1}{(q)_{\infty}} \sum_{n \in \mathbb{Z}} (-1)^nq^{kn^2+ (k-1)n} \\
&= \frac{(q;q^{2k})_{\infty}(q^{2k-1};q^{2k})_{\infty}(q^{2k};q^{2k})_{\infty}}{(q)_{\infty}},
\end{align*}
by the triple product identity \eqref{jtp}.   
\end{proof}

\begin{proof}[Proof of \eqref{family1}]
We begin by inserting the Bailey pair in \eqref{family1alpha} and \eqref{family1beta} into Lemma \ref{x^2qlemma}.  We obtain
\begin{equation*}
\begin{aligned}
(1-x)\sum_{n_1,\dots ,n_k \geq 0}& \frac{(xq)_{n_k}q^{n_1^2+ \cdots + n_{k-1}^2 + n_1 + \cdots + n_{k-1}}x^{2n_1 + \cdots + 2n_{k-1} + n_k}}{(-xq)_{n_1}}\prod_{i=1}^{k-1} \begin{bmatrix} n_{i+1} \\ n_i \end{bmatrix} \\
&= \sum_{n \geq 0} (-1)^nx^{2kn}q^{kn^2+(k-1)n}(1-x^2q^{2n+1}).
\end{aligned}
\end{equation*}
Now, the right-hand side can be written as 
\begin{equation*}
\sum_{n \geq 0} \chi_{4k}(n)q^{\frac{n^2 - (k-1)^2}{4k}}x^{n - (k-1)},
\end{equation*} 
where $\chi_{4k}(n)$ is defined in \eqref{char1}.
%the even periodic function modulo $4k$ defined by
%\begin{equation}
%\chi_{4k}(n) = 
%\begin{cases}
%1, &\text{if $n \equiv k-1$ or $3k+1 \pmod{4k}$}, \\
%-1, &\text{if $n \equiv k+1$ or $3k-1 \pmod{4k}$}, \\
%0, &\text{otherwise}.
%\end{cases}
%\end{equation}
As for the left-hand side, we add and subtract the product $(x)_{\infty}$ and then apply Lemma \ref{qlemma} to obtain
\begin{equation*}
\begin{aligned}
(xq)_{\infty} \sum_{n_1,\dots,n_{k-1} \geq 0}&\frac{q^{n_1^2+ \cdots +n_{k-1}^2 + n_{1} + \cdots + n_{k-1}}x^{2n_1 + \cdots + 2n_{k-1} + n_{k-1}}}{(xq)_{n_{k-1}}(-xq)_{n_1}} \\
&\times \prod_{i=1}^{k-2} \begin{bmatrix} n_{i+1} \\ n_i \end{bmatrix}  \\
+(1-x)\sum_{n_1,\dots,n_k \geq 0}& ((xq)_{n_k} - (xq)_{\infty})\frac{q^{n_1^2+ \cdots +n_{k-1}^2 + n_{1} + \cdots + n_{k-1}}x^{2n_1 + \cdots + 2n_{k-1} + n_k}}{(-xq)_{n_1}} \\
&\times \prod_{i=1}^{k-1} \begin{bmatrix} n_{i+1} \\ n_i \end{bmatrix} \\
= \sum_{n \geq 0} &\chi_{4k}(n)q^{\frac{n^2 - (k-1)^2}{4k}}x^{n - (k-1)}. 
\end{aligned}
\end{equation*}
%Note that when $x=1$ the sum on the left-hand side and the first term on the right-hand side are the same, namely
%\begin{equation}
%(q,q^{2k-1},q^{2k};q^{2k})_{\infty}.
%\end{equation}
%({\bf Not sure if we need to change the base case...})
Taking $\frac{d}{dx}|_{x=1}$ on both sides and using Corollary \ref{family1RR} gives
\begin{align*}
(q)_{\infty}\sum_{n_1,\dots,n_{k-1} \geq 0}& \frac{q^{n_1^2+ \cdots +n_{k-1}^2 + n_{1} + \cdots + n_{k-1}}}{(q)_{n_{k-1}}(-q)_{n_1}} \prod_{i=1}^{k-2} \begin{bmatrix} n_{i+1} \\ n_i \end{bmatrix}\\
&\times \left(2n_1 +\cdots + 2n_{k-1} + n_{k-1}+ \sum_{j=1}^{n_{k-1}} \frac{q^j}{1-q^j} - \sum_{j=1}^{n_1} \frac{q^j}{1+q^j} \right) \\
+ (q;q^{2k})_{\infty}&(q^{2k-1};q^{2k})_{\infty}(q^{2k};q^{2k})_{\infty}\left( k-1 - \sum_{j \geq 1} \frac{q^j}{1-q^j}\right) \\
-\sum_{n_1,\dots,n_k \geq 0}& ((q)_{n_k} - (q)_{\infty})\frac{q^{n_1^2+ \cdots +n_{k-1}^2 + n_{1} + \cdots + n_{k-1}}}{(-q)_{n_1}}  \prod_{i=1}^{k-1} \begin{bmatrix} n_{i+1}  \\ n_i \end{bmatrix} \\
&=  \sum_{n \geq 0} n\chi_{4k}(n)q^{\frac{n^2 - (k-1)^2}{4k}}.
\end{align*}
Letting $q$ approach an odd root of unity we have the strange identity
\begin{equation*}
\begin{aligned}
\sum_{n_1,\dots,n_k \geq 0}& (q)_{n_k}\frac{q^{n_1^2+ \cdots +n_{k-1}^2 + n_{1} + \cdots + n_{k-1}}}{(-q)_{n_1}} \prod_{i=1}^{k-1} \begin{bmatrix} n_{i+1} \\ n_i \end{bmatrix} \\
&``=" -\sum_{n \geq 0} n\chi_{4k}(n)q^{\frac{n^2 - (k-1)^2}{4k}}.
\end{aligned}
\end{equation*}
This is \eqref{family1} for $k \geq 2$.
\end{proof}

\subsection{Proof of \eqref{family2}}
Again we begin with a Bailey pair.
\begin{proposition}
%In \eqref{Slaterpairalpha} and \eqref{Slaterpairbeta}, let $q=q^2$, $a=x^2q^2$, $b = -xq$, and $c=-xq^2$.    The resulting Bailey pair relative to $%(x^2q^2,q^2)$ is
%\begin{equation}
%\alpha_n = \frac{(x^2q^2;q^2)_n(-1)^nq^{n^2}(1-xq^{2n+1})}{(q^2;q^2)_n(1-x^2q^2)}
%\end{equation}  
%and
%\begin{equation}
%\beta_n = \frac{(q;q^2)_n}{(q^2;q^2)_n(-xq)_{2n+1}}.
%\end{equation}
For $k \geq 2$, the sequences $(\alpha_n,\beta_n)$ form a Bailey pair relative to $(x^2q^2,q^2)$, where
\begin{equation} \label{family2alpha}
\alpha_n = \frac{(x^2q^2;q^2)_n(-1)^nx^{(2k-2)n}q^{(2k-1)n^2+(2k-2)n}(1-xq^{2n+1})}{(q^2;q^2)_n(1-x^2q^2)}
\end{equation}
and
\begin{equation} \label{family2beta}
\beta_n = \beta_{n_k} = \sum_{n_1,\dots,n_{k-1} \geq 0}\frac{q^{2n_1^2+2n_1 \cdots + 2n_{k-1}^2 + 2n_{k-1}}x^{2n_1 + \cdots + 2n_{k-1}}(q;q^2)_{n_1}}{(q^2;q^2)_{n_k}(-xq)_{2n_1+1}}\prod_{i=1}^{k-1} \begin{bmatrix} n_{i+1} \\ n_i \end{bmatrix}_{q^2}.
\end{equation}.
\end{proposition}

\begin{proof}
This follows from iterating the Bailey pair in \eqref{BPfamily2basealpha} and \eqref{BPfamily2basebeta} $k-1$ times in \eqref{alphaprimedef} and \eqref{betaprimedef} with $\rho_1, \rho_2 \to \infty$. 
\end{proof}

\begin{corollary} \label{family2RR}
Recall the periodic function $\chi_{8k-4}(n)$ defined in \eqref{char2}.  For  $k \geq 2$ we have the identities
\begin{align*}
\sum_{n_1,\dots,n_{k-1} \geq 0} &\frac{q^{2n_1^2+2n_1 \cdots + 2n_{k-1}^2 + 2n_{k-1}}(q;q^2)_{n_1}}{(q^2;q^2)_{n_{k-1}}(-xq)_{2n_1+1}}\prod_{i=1}^{k-2} \begin{bmatrix} n_{i+1} \\ n_i \end{bmatrix}_{q^2} \\
&= \frac{1}{(q^2;q^2)_{\infty}}\sum_{n \geq 0} \chi_{8k-4}(n)q^{\frac{n^2 - (2k-2)^2}{8k-4}} \\
&= \frac{(q;q^{4k-2})_{\infty}(q^{4k-3};q^{4k-2})_{\infty}(q^{4k-2};q^{4k-2})_{\infty}}{(q^2;q^2)_{\infty}}. 
\end{align*}
\end{corollary}

\begin{proof}
We use the Bailey pair in \eqref{family2alpha} and \eqref{family2beta} with $x=1$ in the definition of a Bailey pair \eqref{pairdef} and let $n \to \infty$.    This gives
\begin{align*}
\sum_{n_1,\dots,n_{k-1} \geq 0} &\frac{q^{2n_1^2+2n_1 \cdots + 2n_{k-1}^2 + 2n_{k-1}}(q;q^2)_{n_1}}{(q^2;q^2)_{n_{k-1}}(-xq)_{2n_1+1}}\prod_{i=1}^{k-2} \begin{bmatrix} n_{i+1} \\ n_i \end{bmatrix}_{q^2} \\
&= \frac{1}{(q^2;q^2)_{\infty}} \sum_{n \geq 0} (-1)^nq^{(2k-1)n^2 + (2k-2)n}(1-q^{2n+1})  \\
&= \frac{1}{(q^2;q^2)_{\infty}} \sum_{n \geq 0} \chi_{8k-4}(n)q^{\frac{n^2 - (2k-2)^2}{8k-4}} \\
&= \frac{1}{(q^2;q^2)_{\infty}} \sum_{n \in \mathbb{Z}} (-1)^nq^{(2k-1)n^2+ (2k-2)n} \\
&= \frac{(q;q^{4k-2})_{\infty}(q^{4k-3};q^{4k-2})_{\infty}(q^{4k-2};q^{4k-2})_{\infty}}{(q^2;q^2)_{\infty}},
\end{align*}
by the triple product identity \eqref{jtp}.   
\end{proof}

\begin{proof}[Proof of \eqref{family2}]
Inserting the Bailey pair from \eqref{family2alpha} and \eqref{family2beta} into Lemma \ref{x^2qlemma}, we obtain
\begin{equation*}
\begin{aligned}
(1-x)\sum_{n_1,\dots,n_{k} \geq 0}&\frac{(xq^2;q^2)_{n_k}q^{2n_1^2+2n_1 \cdots + 2n_{k-1}^2 + 2n_{k-1}}x^{2n_1 + \cdots + 2n_{k-1}+n_k}(q;q^2)_{n_1}}{(-xq)_{2n_1+1}}\prod_{i=1}^{k-1} \begin{bmatrix} n_{i+1} \\ n_i \end{bmatrix}_{q^2} \\
&= \sum_{n \geq 0} (-1)^n(1-xq^{2n+1})x^{(2k-1)n}q^{(2k-1)n^2+ (2k-2)n}.
\end{aligned}
\end{equation*}
Note that the right-hand side can be written
\begin{equation*}
\sum_{n \geq 0} \chi_{8k-4}(n)q^{\frac{n^2 - (2k-2)^2}{8k-4}}x^{\frac{n - (2k-2)}{2}},
\end{equation*} 
where $\chi_{8k-4}(n)$ is defined in \eqref{char1}.
%the even periodic function modulo $8k-4$ defined by
%\begin{equation}
%\chi_{8k-4}(n) = 
%\begin{cases}
%1, &\text{if $n \equiv 2k-2$ or $6k-2\pmod{8k-4}$}, \\
%-1, &\text{if $n \equiv 2k$ or $6k-4 \pmod{8k-4}$}, \\
%0, &\text{otherwise}.
%\end{cases}
%\end{equation}
Adding and subtracting $(x;q^2)_{\infty}$ on the left-hand side and applying Lemma \ref{qlemma}, we have 
\begin{equation*}
\begin{aligned}
(xq^2;q^2)_{\infty} \sum_{n_1,\dots,n_{k-1} \geq 0}&\frac{q^{2n_1^2+2n_1+ \cdots + 2n_{k-1}^2 + 2n_{k-1}}x^{2n_1 + \cdots + 2n_{k-1} + n_{k-1}}(q;q^2)_{n_1}}{(xq^2;q^2)_{n_{k-1}}(-xq)_{2n_1+1}}  \prod_{i=1}^{k-2} \begin{bmatrix} n_{i+1} \\ n_i \end{bmatrix}_{q^2}  \\
+(1-x)\sum_{n_1,\dots,n_k \geq 0}& ((xq^2;q^2)_{n_k} - (xq^2;q^2)_{\infty}) \\
&\times \frac{q^{2n_1^2+2n_1+ \cdots +2 n_{k-1}^2 + 2n_{k-1}}x^{2n_1 + \cdots + 2n_{k-1} + n_k}(q;q^2)_{n_1}}{(-xq)_{2n_1+1}} \prod_{i=1}^{k-1} \begin{bmatrix} n_{i+1} \\ n_i \end{bmatrix}_{q^2} \\
= \sum_{n \geq 0}& \chi_{8k-4}(n)q^{\frac{n^2 - (2k-2)^2}{8k-4}}x^{\frac{n - (2k-2)}{2}}.
\end{aligned}
\end{equation*}
Taking $\frac{d}{dx}|_{x=1}$ on both sides and using Corollary \ref{family2RR} gives
\begin{align*}
(q^2;q^2)_{\infty}\sum_{n_1,\dots,n_{k-1} \geq 0}& \frac{(q;q^2)_{n_1}q^{2n_1^2+ \cdots +2n_{k-1}^2 + 2n_{1} + \cdots + 2n_{k-1}}}{(q^2;q^2)_{n_{k-1}}(-q)_{2n_1+1}} \prod_{i=1}^{k-2} \begin{bmatrix} n_{i+1} \\ n_i \end{bmatrix}_{q^2}\\
&\times \left(2n_1 +\cdots + 2n_{k-1} + n_{k-1}+ \sum_{j=1}^{n_{k-1}} \frac{q^{2j}}{1-q^{2j}} - \sum_{j=1}^{2n_1+1} \frac{q^j}{1+q^j} \right) \\
+ (q;q^{4k-2})_{\infty}&(q^{4k-3};q^{4k-2})_{\infty}(q^{4k-2};q^{4k-2})_{\infty}\left( k-1 - \sum_{j \geq 1} \frac{q^{2j}}{1-q^{2j}}\right) \\
-\sum_{n_1,\dots,n_k \geq 0}& ((q^2;q^2)_{n_k} - (q^2;q^2)_{\infty})\frac{q^{2n_1^2+ \cdots +2n_{k-1}^2 + 2n_{1} + \cdots + 2n_{k-1}}(q;q^2)_{n_1}}{(-q)_{2n_1+1}}  \prod_{i=1}^{k-1} \begin{bmatrix} n_{i+1}  \\ n_i \end{bmatrix}_{q^2} \\
&= \frac{1}{2} \sum_{n \geq 0} n\chi_{8k-4}(n)q^{\frac{n^2 - (2k-2)^2}{8k-4}}.
\end{align*}
Letting $q$ tend to an odd root of unity gives the strange identity
\begin{equation*}
\begin{aligned}
\sum_{n_1,\dots,n_k \geq 0}& (q^2;q^2)_{n_k}\frac{q^{2n_1^2+2n_1+ \cdots +2n_{k-1}^2 + 2n_{k-1}}(q;q^2)_{n_1}}{(-q)_{2n_1+1}} \prod_{i=1}^{k-1} \begin{bmatrix} n_{i+1} \\ n_i \end{bmatrix}_{q^2} \\
&``=" -\frac{1}{2}\sum_{n \geq 0} n\chi_{8k-4}(n)q^{\frac{n^2 - (2k-2)^2}{8k-4}}.
\end{aligned}
\end{equation*}
This is \eqref{family2}
\end{proof}

\subsection{Proof of \eqref{family3}}
Unlike the Bailey pairs in the previous two subsections, the ones here and in the next subsection depend on two parameters, $k$ and $a$.
\begin{proposition} \label{thirdpairprop}
For $k \geq 2$ and $0 \leq a \leq k-1$, the following is a Bailey pair relative to $(x^2q^2,q^2)$:
\begin{equation} \label{thirdalpha}
\alpha_n = \frac{(x^2q^2;q^2)_n}{(q^2;q^2)_n(1-x^2q^2)}(-1)^nx^{(2k-1)n}q^{2(a+1)n^2+n + 2(k-a-1)(n^2+n)}(1+x^{2a+1}q^{(2a+1)(2n+1)}).
\end{equation}
and
\begin{equation} \label{thirdbeta}
\beta_n = \beta_{n_k} = \sum_{n_1,n_2,\dots,n_{k-1} \geq 0} \frac{q^{2n_1^2+ \cdots + 2n_{k-1}^2 + 2n_{a+1} + \cdots + 2n_{k-1}}x^{2n_1 + \cdots + 2n_{k-1}}}{(q^2;q^2)_{n_k}(-xq;q^2)_{n_1 + \delta_{a,0}}} \prod_{i=1}^{k-1} \begin{bmatrix} n_{i+1} +\delta_{i,a} \\ n_i \end{bmatrix}_{q^2} 
\end{equation}
\end{proposition}
%{\bf Need to add a condition for when $a=0$.   Namely, the $(xq;q^2)_{n_1}$ becomes $(xq;q^2)_{n_1+1}$.}

%{\bf Checks out in MAPLE, up to the above and comparing powers of $x$}

\begin{proof}
We take $q=q^2$, $a=x^2$, $b = -xq$, and $c \to \infty$ in \eqref{Slaterpairalpha} and \eqref{Slaterpairbeta}.   This gives the Bailey pair relative to $(x^2,q^2)$,
\begin{equation*}
\alpha_n = \frac{(x^2;q^2)_n(1-x^2q^{4n})(-1)^nx^nq^{2n^2-n}}{(q^2;q^2)_n(1-x^2)}
\end{equation*}    
and
\begin{equation*}
\beta_n = \frac{1}{(q^2;q^2)_n(-xq;q^2)_n}.
\end{equation*}
Assuming for a moment that $a \neq 0$, we use this pair in \eqref{alphaprimedef} and \eqref{betaprimedef} and iterate $a$ times with $\rho_1, \rho_2 \to \infty$.   The result is the Bailey pair relative to $(x^2,q^2)$,
\begin{equation*}
\alpha_n = \frac{(x^2;q^2)_n(1-x^2q^{4n})(-1)^nx^{(2a+1)n}q^{2(a+1)n^2-n}}{(q^2;q^2)_n(1-x^2)}
\end{equation*} 
and
\begin{equation*}
\beta_n =\beta_{n_{a+1}} = \sum_{n_1,\dots,n_a \geq 0} \frac{q^{2n_1^2+ \cdots + 2n_a^2}x^{2n_1 + \cdots + 2n_a}}{(q^2;q^2)_{n_{a+1} - n_a} \cdots (q^2;q^2)_{n_1}(-xq;q^2)_{n_1}}.
\end{equation*}
Applying Lemma \ref{aux3} and computing as in the proof of Proposition \ref{Hikamipair}, we obtain a Bailey pair relative to $(x^2q^2,q^2)$,
\begin{equation*}
\alpha_n =\frac{1}{1-x^2q^2}\left(\frac{(x^2q^2;q^2)_n(-1)^nx^{(2a+1)n}q^{(2a+2)n^2+n}(1-x^{2a+1}q^{(2a+1)(2n+1)})}{(q^2;q^2)_n}\right)
\end{equation*}
and 
\begin{equation} \label{family3betaafteraiterations}
\beta_n = \beta_{n_{a+1}} =  \sum_{n_1,\dots,n_a \geq 0} \frac{q^{2n_1^2+ \cdots + 2n_a^2}x^{2n_1 + \cdots + 2n_a}(1-q^{2n_{a+1}+2})}{(q^2;q^2)_{n_{a+1} + 1 - n_a} \cdots (q^2;q^2)_{n_1}(-xq;q^2)_{n_1}}.
\end{equation}
Iterating $k=1=a$ times using \eqref{alphaprimedef} and \eqref{betaprimedef} with $\rho_1, \rho_2 \to \infty$ and then multiplying the numerator and denominator of the resulting $\beta_n$ by
\begin{equation*}
(q^2;q^2)_{n_k} \cdots (q^2;q^2)_{n_2}
\end{equation*}
gives the result for $a \geq 1$.   When $a=0$ we have 
\begin{equation*}
\beta_n = \beta_{n_1} = \frac{1}{(q^2;q^2)_n(-xq;q^2)_{n_1+1}}
\end{equation*} 
at \eqref{family3betaafteraiterations}, and the rest of the proof is similar.
\end{proof}

%As a corollary, we have some ``variants" of Rogers-Ramanujan type identities, analogous to Corollary \ref{AGvariant}    
\begin{corollary} \label{family3RR}
Recall the definition of $\chi_{8k}^{(a)}(n)$ from \eqref{char2}.   For $k \geq 2$ we have the identities
\begin{align*}
\sum_{n_1,n_2,\dots,n_{k-1} \geq 0}& \frac{q^{2n_1^2+ \cdots + 2n_{k-1}^2 + 2n_{a+1} + \cdots + 2n_{k-1}}}{(q^2;q^2)_{n_{k-1}}(-q;q^2)_{n_1 + \delta_{a,0}}} \prod_{i=1}^{k-2} \begin{bmatrix} n_{i+1} + \delta_{i,a}\\ n_i \end{bmatrix}_{q^2} \\
&= \frac{1}{(q^2;q^2)_{\infty}} \sum_{n \geq 0} \chi_{8k}^{(a)}(n)q^{\frac{n^2 - (2k-2a-1)^2}{8k}} \\
&= \frac{(q^{2a+1},q^{4k-2a-1},q^{4k};q^{4k})_{\infty}}{(q^2;q^2)_{\infty}}. 
\end{align*}
\end{corollary}

\begin{proof}
We take the Bailey pair from \eqref{thirdalpha} and \eqref{thirdbeta} with $x=1$, put it in the definition of a Bailey pair, and let $n \to \infty$.   Then we have
\begin{align*}
\sum_{n_1,n_2,\dots,n_{k-1} \geq 0}& \frac{q^{2n_1^2+ \cdots + 2n_{k-1}^2 + 2n_{a+1} + \cdots + 2n_{k-1}}}{(q^2;q^2)_{n_{k-1}}(-q;q^2)_{n_1 + \delta_{a,0}}} \prod_{i=1}^{k-2} \begin{bmatrix} n_{i+1} + \delta_{i,a}\\ n_i \end{bmatrix}_{q^2}  \\
&= \frac{1}{(q^2;q^2)_{\infty}} \times \sum_{n \geq 0} (-1)^nq^{2kn^2 + (2k-2a-1)n}(1+q^{(2a+1)(2n+1)}) \\
&=  \frac{1}{(q^2;q^2)_{\infty}} \sum_{n \geq 0} \chi_{8k}^{(a)}(n)q^{\frac{n^2 - (2k-2a-1)^2}{8k}} \\
&= \frac{1}{(q^2;q^2)_{\infty}} \times \sum_{n \in \mathbb{Z}} (-1)^nq^{2kn^2 + (2k-2a-1)n} \\
&= \frac{1}{(q^2;q^2)_{\infty}} \times (q^{2a+1},q^{4k-2a-1},q^{4k};q^{4k})_{\infty}.
\end{align*}
\end{proof}

\begin{proof}[Proof of \eqref{family3}]
Using the Bailey pair in \eqref{thirdalpha} and \eqref{thirdbeta} in Lemma \ref{x^2qlemma} gives
\begin{equation*}
\begin{aligned}
(1-x)\sum_{n_1,\dots ,n_k \geq 0}& \frac{(xq^2;q^2)_{n_k}q^{2n_1^2+ \cdots + 2n_{k-1}^2 + 2n_{a+1} + \cdots + 2n_{k-1}}x^{2n_1 + \cdots + 2n_{k-1} + n_k}}{(-xq;q^2)_{n_1+\delta_{a,0}}}\prod_{i=1}^{k-1} \begin{bmatrix} n_{i+1} + \delta_{i,a} \\ n_i \end{bmatrix}_{q^2} \\
&= \sum_{n \geq 0} (-1)^nx^{2kn}q^{2kn^2+(2k-2a-1)n}(1-x^{2a+1}q^{(2a+1)(2n+1)}).
\end{aligned}
\end{equation*}
Now, the right-hand side can be written as 
\begin{equation*}
\sum_{n \geq 0} \chi_{8k}^{(a)}(n)q^{\frac{n^2 - (2k-2a-1)^2}{8k}}x^{\frac{n - (2k-2a-1)}{2}},
\end{equation*} 
where $\chi_{8k}^{(a)}(n)$ is defined in \eqref{char2}.
%the even periodic function modulo $8k$ defined by
%\begin{equation}
%\chi_{8k}^{(a)}(n) = 
%\begin{cases}
%1, &\text{if $n \equiv 2k-2a-1$ or $6k+2a+1 \pmod{8k}$}, \\
%-1, &\text{if $n \equiv 2k+2a+1$ or $6k-2a-1 \pmod{8k}$}, \\
%0, &\text{otherwise}.
%\end{cases}
%\end{equation}
Adding and subtracting $(x;q^2)_{\infty}$ on the left-hand side and applying Lemma \ref{qlemma} gives
\begin{equation*}
\begin{aligned}
 (xq^2;q^2)_{\infty} \sum_{n_1,\dots,n_{k-1} \geq 0}&\frac{q^{2n_1^2+ \cdots +2n_{k-1}^2 + 2n_{a+1} + \cdots + 2n_{k-1}}x^{2n_1 + \cdots + 2n_{k-1} + n_{k-1} - \delta_{a,k-1}\chi(n_{k-1} > 0)}}{(xq^2;q^2)_{n_{k-1}}(-xq;q^2)_{n_1 + \delta_{a,0}}} \\
&\times \prod_{i=1}^{k-1} \begin{bmatrix} n_{i+1} + \delta_{i,a}\\ n_i \end{bmatrix}_{q^2} \\
+(1-x)\sum_{n_1,\dots,n_k \geq 0}& ((xq^2;q^2)_{n_k} - (xq^2;q^2)_{\infty}) \\
&\times \frac{q^{2n_1^2+ \cdots +2n_{k-1}^2 + 2n_{a+1} + \cdots + 2n_{k-1}}x^{2n_1 + \cdots + 2n_{k-1} + n_k}}{(-xq;q^2)_{n_1 + \delta_{a,0}}} \prod_{i=1}^{k-1} \begin{bmatrix} n_{i+1} + \delta_{i,a} \\ n_i \end{bmatrix}_{q^2} \\
= \sum_{n \geq 0}& \chi_{8k}^{(a)}(n)q^{\frac{n^2 - (2k-2a-1)^2}{8k}}x^{\frac{n - (2k-2a-1)}{2}}.
\end{aligned}
\end{equation*}
Taking $\frac{d}{dx}|_{x=1}$ on both sides and using Corollary \ref{family3RR} gives 
\begin{align*}
(q^2;q^2)_{\infty}\sum_{n_1,\dots,n_{k-1} \geq 0}& \frac{q^{2n_1^2+ \cdots +2n_{k-1}^2 + 2n_{1} + \cdots + 2n_{k-1}}}{(q^2;q^2)_{n_{k-1}}(-q;q^2)_{n_1+\delta_{a,0}}} \prod_{i=1}^{k-2} \begin{bmatrix} n_{i+1} \\ n_i \end{bmatrix}_{q^2}\\ 
\times \Bigg(2n_1 &+\cdots + 2n_{k-1} + n_{k-1} - \delta_{a,k-1}\chi(n_{k-1} > 0)  \\
&+ \sum_{j=1}^{n_{k-1}} \frac{q^{2j}}{1-q^{2j}} - \sum_{j=1}^{n_1+\delta_{a,0}} \frac{q^{2j-1}}{1+q^{2j-1}} \Bigg) \\
+ (q^{2a+1};q^{4k})_{\infty}&(q^{4k-2a-1};q^{4k})_{\infty}(q^{4k};q^{4k})_{\infty}\left( \frac{2k-2a-1}{2} - \sum_{j \geq 1} \frac{q^{2j}}{1-q^{2j}}\right) \\
-\sum_{n_1,\dots,n_k \geq 0}& ((q^2;q^2)_{n_k} - (q^2;q^2)_{\infty})\frac{q^{2n_1^2+ \cdots +2n_{k-1}^2 + 2n_{1} + \cdots + 2n_{k-1}}}{(-q;q^2)_{n_1+\delta_{a,0}}}  \prod_{i=1}^{k-1} \begin{bmatrix} n_{i+1}  \\ n_i \end{bmatrix}_{q^2} \\
&=  \frac{1}{2}\sum_{n \geq 0} n\chi_{8k}(n)q^{\frac{n^2 - (2k-2a-1)^2}{8k}}.
\end{align*}
As $q$ tends to an odd root of unity, this implies the strange identity
\begin{equation*}
\begin{aligned}
\sum_{n_1,\dots,n_k \geq 0}& (q^2;q^2)_{n_k}\frac{q^{2n_1^2+ \cdots +2n_{k-1}^2 + 2n_{a+1} + \cdots + 2n_{k-1}}}{(-q;q^2)_{n_1+\delta_{a,0}}} \prod_{i=1}^{k-1} \begin{bmatrix} n_{i+1}  + \delta_{i,a}\\ n_i \end{bmatrix}_{q^2} \\
&``=" -\frac{1}{2}\sum_{n \geq 0} n\chi_{8k}^{(a)}(n)q^{\frac{n^2 - (2k-2a-1)^2}{8k}}.
\end{aligned}
\end{equation*}
This is \eqref{family3}.

\end{proof}

\subsection{Proof of \eqref{family4}}
As usual, we begin with a Bailey pair.
\begin{proposition} \label{fourthpairprop}
For $k \geq 2$ and $0 \leq a \leq k-1$, the following is a Bailey pair relative to $x^2q$:
\begin{equation} \label{fourthalpha}
\alpha_n = \frac{(x^2q)_n}{(q)_n(1-x^2q)}x^{(2k-2)n}q^{\binom{n+1}{2} + an^2 + (k-a-1)(n^2+n)}(1+x^{2a}q^{a(2n+1)})
\end{equation}
and
\begin{equation} \label{fourthbeta}
= \sum_{n_1,\dots,n_{k-1} \geq 0} \frac{q^{n_1^2+ \cdots + n_{k-1}^2 + n_{a+1} + \cdots + n_{k-1}x^{2n_1 + \cdots +2n_{k-1}}}(-1)_{n_1 + \delta_{a,0}}}{(q)_{n_k}(x^2q;q^2)_{n_1 + \delta_{a,0}}} \prod_{i=1}^{k-1} \begin{bmatrix} n_{i+1} + \delta_{i,a} \\ n_i \end{bmatrix}.
\end{equation}
\end{proposition}

\begin{proof}
Let $b = -c = xq^{1/2}$ and $a=x^2$ in \eqref{Slaterpairalpha} and \eqref{Slaterpairbeta}.   This gives the Bailey pair relative to $x^2$,
\begin{equation*}
\alpha_n = \frac{(x^2)_n(1-x^2q^{2n})q^{\binom{n}{2}}}{(q)_n(1-x^2)}
\end{equation*}
and
\begin{equation*}
\beta_n = \frac{(-1)_n}{(q)_n(x^2q;q^2)_n}.
\end{equation*}
Iterating \eqref{alphaprimedef} and \eqref{betaprimedef} $a$ times, for $a \geq 1$, with $\rho_1, \rho_2 \to \infty$, we obtain
\begin{equation*}
\alpha_n = \frac{(x^2)_n(1-x^2q^{2n})x^{2an}q^{\binom{n}{2} + an^2}}{(q)_n(1-x^2)}
\end{equation*}
and
\begin{equation*}
\beta_n = \beta_{n_{a+1}} = \sum_{n_1,\dots, n_{a}  \geq 0} \frac{q^{n_1^2+\cdots +n_{a}^2}x^{2n_1 + \cdots + 2n_a}(-1)_{n_1}}{(q)_{n_{a+1} - n_a} \cdots (q)_{n_2-n_1}(q)_{n_1}(x^2;q^2)_{n_1}}.
\end{equation*} 
Applying Lemma \ref{aux3} and performing a short calculation as in the proof of Proposition \ref{Hikamipair} or Proposition \ref{thirdpairprop} gives a Bailey pair relative to $x^2q$,
\begin{equation*}
\alpha_n = \frac{(x^2)_n}{(q)_n(1-x^2q)}x^{2an}q^{\binom{n+1}{2} + an^2}(1+x^{2a}q^{a(2n+1)})
\end{equation*}
and
\begin{equation} \label{fourthalphaafteraiterations}
\beta_n = \beta_{n_{a+1}}  = \sum_{n_1,\dots, n_{a}  \geq 0} \frac{q^{n_1^2+\cdots +n_{a}^2}x^{2n_1 + \cdots + 2n_a}(1-q^{n_{a+1}+1})(-1)_{n_1}}{(q)_{n_{a+1}+1 - n_a} \cdots (q)_{n_2-n_1}(q)_{n_1}(x^2q;q^2)_{n_1}}.
\end{equation}
Finally, we iterate this $k-1-a$ times along the Bailey chain in \eqref{alphaprimedef} and \eqref{betaprimedef} with $\rho_1,\rho_2 \to \infty$ to obtain the desired $\alpha_n$ and
\begin{equation*}
\begin{aligned}
&\beta_n = \beta_{n_k} = \\
&\sum_{n_1,\dots, n_{k-1}  \geq 0} \frac{q^{n_1^2+\cdots +n_{k-1}^2 + n_{a+1} + \cdots + n_{k-1}}x^{2n_1 + \cdots + 2n_{k-1}}(1-q^{n_{a+1}+1})(-1)_{n_1}}{(q)_{n_k - n_{k-1}} \cdots (q)_{n_{a+2} - n_{a+1}}(q)_{n_{a+1}+1 - n_a}(q)_{n_a - n_{a-1}} \cdots (q)_{n_2-n_1}(q)_{n_1}(x^2q;q^2)_{n_1}}. 
\end{aligned}
\end{equation*}
Converting to $q$-binomial notation gives the result.    For $a = 0$ we have 
\begin{equation*}
\beta_n = \beta_{n_1} = \frac{(-1)_{n_1+1}}{(q)_{n_1}(x^2q;q^2)_{n_1 + 1}}
\end{equation*}
at \eqref{fourthalphaafteraiterations} and then the rest of the proof is the same.
\end{proof}

%We need some Andrews-Gordon type identities.
\begin{corollary} \label{family4RR}
Recall the definition of $\chi_{4k-2}^{(a)}$ in \eqref{char3}.   For $k \geq 2$ and $0 \leq a \leq k-1$ we have the identities
\begin{align*}
= \sum_{n_1,\dots,n_{k-1} \geq 0}& \frac{q^{n_1^2+ \cdots + n_{k-1}^2 + n_{a+1} + \cdots + n_{k-1}}(-1)_{n_1 + \delta_{a,0}}}{(q)_{n_{k-1}}(q;q^2)_{n_1 + \delta_{a,0}}} \prod_{i=1}^{k-2} \begin{bmatrix} n_{i+1} + \delta_{i,a} \\ n_i \end{bmatrix} \\
&= (1+\delta_{a,0})\frac{1}{(q)_{\infty}}\sum_{n \geq 0} \chi_{4k-2}^{(a)}(n)q^{\frac{n^2 - (2k-2a-1)^2}{8(2k-1)}} \\
&= \frac{(q^{2k-1};q^{2k-1})_{\infty}(-q^a;q^{2k-1})_{\infty}(-q^{2k-a-1};q^{2k-1})_{\infty}}{(q)_{\infty}}.
\end{align*}
\end{corollary}

\begin{proof}
We take the Bailey pair from \eqref{fourthalpha} and \eqref{fourthbeta} with $x=1$, insert it into the definition of a Bailey pair \eqref{pairdef}, and let $n \to \infty$.   Then we have
\begin{align*}
\sum_{n_1,\dots,n_{k-1} \geq 0}& \frac{q^{n_1^2+ \cdots + n_{k-1}^2 + n_{a+1} + \cdots + n_{k-1}}(-1)_{n_1 + \delta_{a,0}}}{(q)_{n_{k-1}}(q;q^2)_{n_1 + \delta_{a,0}}} \prod_{i=1}^{k-2} \begin{bmatrix} n_{i+1} + \delta_{i,a} \\ n_i \end{bmatrix}  \\
&= \frac{1}{(q)_{\infty}} \times \sum_{n \geq 0} q^{\frac{(2k-1)}{2}n^2 + \frac{2k-2a-1}{2}n}(1+q^{a(2n+1)}) \\
&=  \frac{(1+\delta_{a,0})}{(q)_{\infty}} \sum_{n \geq 0} \chi_{4k-2}^{(a)}(n)q^{\frac{n^2 - (2k-2a-1)^2}{8(2k-1)}} \\
&= \frac{1}{(q)_{\infty}} \times \sum_{n \in \mathbb{Z}} q^{\frac{(2k-1)}{2}n^2 + \frac{2k-2a-1}{2}n} \\
&= \frac{1}{(q)_{\infty}} \times (-q^{a};q^{2k-1})_{\infty}(-q^{2k-a-1};q^{2k-1})_{\infty}(q^{2k-1};q^{2k-1})_{\infty}.
\end{align*}
\end{proof}

\begin{proof}[Proof of \eqref{family4}]
We first use the Bailey pair in Proposition \ref{fourthpairprop} in Lemma \ref{x^2qlemma}, which gives
\begin{equation*}
\begin{aligned}
(1-x)\sum_{n_1,\dots,n_k \geq 0}& \frac{(xq)_{n_k}q^{n_1^2+\cdots +n_{k-1}^2 + n_{a+1} + \cdots + n_{k-1}}(-1)_{n_1 + \delta_{a,0}}x^{2n_1 + \cdots +2n_{k-1}+n_k}}{(x^2q;q^2)_{n_1 + \delta_{a,0}}}\prod_{i=1}^{k-1} \begin{bmatrix} n_{i+1} + \delta_{i,a} \\ n_i \end{bmatrix} \\
&= \sum_{n \geq 0} x^{(2k-1)n}q^{\binom{n+1}{2} + an^2+ (k-a-1)(n^2+n)}(1+x^{2a}q^{a(2n+1)}).
\end{aligned}
\end{equation*}

Note that the right-hand side can be written
\begin{equation*}
\left( 1 + \delta_{a,0} \right) \sum_{n \geq 0} \chi_{4k-2}^{(a)}(n)q^{\frac{n^2 - (2k-2a-1)^2}{8(2k-1)}}x^{\frac{n - (2k-2a-1)}{2}},
\end{equation*} 
where $\chi_{4k-2}^{(a)}(n)$ is defined in \eqref{char3}.
%the even periodic function modulo $8k-4$ defined by
%\begin{equation}
%\chi_{4k-2}^{(a)}(n) = 
%\begin{cases}
%1, &\text{if $n \equiv  \pm (2k-2a-1) \pmod{4k-2}$}, \\
%0, &\text{otherwise}.
%\end{cases}
%\end{equation}

Adding and subtracting $(xq)_{\infty}$ on the left-hand side and applying Lemma \ref{qlemma} we obtain
\begin{equation*}
\begin{aligned}
(xq)_{\infty} \sum_{n_1,\dots,n_{k-1} \geq 0}&\frac{q^{n_1^2+\cdots + n_{k-1}^2 + n_{a+1}+ \cdots + n_{k-1}}x^{2n_1 + \cdots + 2n_{k-1} + n_{k-1} - \delta_{a,k-1}\chi(n_{k-1} > 0)}(-1)_{n_1+\delta_{a,0}}}{(xq;q)_{n_{k-1}}(x^2q;q^2)_{n_1+\delta_{a,0}}} \\
&\times \prod_{i=1}^{k-2} \begin{bmatrix} n_{i+1} + \delta_{i,a} \\ n_i \end{bmatrix}  \\
+(1-x)\sum_{n_1,\dots,n_k \geq 0}& ((xq)_{n_k} - (xq)_{\infty})\frac{q^{n_1^2 \cdots + n_{k-1}^2 + n_{a+1} + \cdots + n_{k-1}}x^{2n_1 + \cdots + 2n_{k-1} + n_k}(-1)_{n_1+\delta_{a,0}}}{(x^2q;q^2)_{n_1+\delta_{a,0}}} \\
&\times \prod_{i=1}^{k-1} \begin{bmatrix} n_{i+1} + \delta_{i,a} \\ n_i \end{bmatrix} \\
= (1+\delta_{a,0})&\sum_{n \geq 0} \chi_{4k-2}^{(a)}(n)q^{\frac{n^2 - (2k-2a-1)^2}{8(2k-1)}}x^{\frac{n - (2k-2a-1)}{2}}.
\end{aligned}
\end{equation*}
Taking $\frac{d}{dx}|{x=1}$ on both sides and using Corollary \ref{family4RR} we find
\begin{align*}
(q)_{\infty}\sum_{n_1,\dots,n_{k-1} \geq 0}& \frac{q^{n_1^2+ \cdots +n_{k-1}^2 + n_{a+1} + \cdots + n_{k-1}}(-1)_{n_1 + \delta_{a,0}}}{(q)_{n_{k-1}}(q;q^2)_{n_1+\delta_{a,0}}} \prod_{i=1}^{k-2} \begin{bmatrix} n_{i+1}+\delta_{a,i} \\ n_i \end{bmatrix}\\
\times \Bigg(2n_1 &+\cdots + 2n_{k-1} + n_{k-1} - \delta_{a,k-1}\chi(n_{k-1} > 0) \\
&+ \sum_{j=1}^{n_{k-1}} \frac{q^j}{1-q^j} + 2\sum_{j=1}^{n_1+\delta_{a,0}} \frac{q^{2j-1}}{1-q^{2j-1}} \Bigg) \\
+  (-q^a;q^{2k-1})_{\infty}&(-q^{2k-a-1};q^{2k-1})_{\infty}(q^{2k-1};q^{2k-1})_{\infty}\left( \frac{2k-2a-1}{2} - \sum_{j \geq 1} \frac{q^j}{1-q^j}\right) \\
-\sum_{n_1,\dots,n_k \geq 0}& ((q)_{n_k} - (q)_{\infty})\frac{q^{n_1^2+n_2^2+ \cdots +n_{k-1}^2 + n_{a+1} + \cdots + n_{k-1}}(-1)_{n_1 + \delta_{a,0}}}{(q;q^2)_{n_1 + \delta_{a,0}}}  \prod_{i=1}^{k-1} \begin{bmatrix} n_{i+1} +\delta_{a,i} \\ n_i \end{bmatrix} \\
&=  (1+\delta_{a,0})\frac{1}{2}\sum_{n \geq 0} n\chi_{4k-2}(n)q^{\frac{n^2 - (2k-2a-1)^2}{8(2k-1)}}.
\end{align*}
Letting $q$ tend to an even root of unity, we obtain the strange identity in \eqref{family4},
\begin{equation*}
\begin{aligned}
\sum_{n_1,\dots,n_k \geq 0}& (q)_{n_k}\frac{q^{n_1^2+ \cdots +n_{k-1}^2 + n_{a+1} + \cdots + n_{k-1}}(-1)_{n_1+\delta_{a,0}}}{(q;q^2)_{n_1+\delta_{a,0}}} \prod_{i=1}^{k-1} \begin{bmatrix} n_{i+1} + \delta_{i,a} \\ n_i \end{bmatrix} \\
&``=" -(1+\delta_{a,0})\frac{1}{2}\sum_{n \geq 0} n\chi_{4k-2}^{(a)}(n)q^{\frac{n^2 - (2k-2a-1)^2}{8(2k-1)}}.
\end{aligned}
\end{equation*}
\end{proof}

\section{Conclusion}
We conclude with a few remarks.  First comparing \eqref{family1} and \eqref{family2} gives a ``quantum $q$-series identity" in the sense of \cite{Lo.5},
\begin{equation*}
\begin{aligned}
\sum_{n_1,\dots,n_{2k-1} \geq 0} &(q)_{n_{2k-1}}\frac{q^{n_1^2+ \cdots +n_{2k-2}^2 + n_{1} + \cdots + n_{2k-2}}}{(-q)_{n_1}} \prod_{i=1}^{2k-2} \begin{bmatrix} n_{i+1} \\ n_i \end{bmatrix} \\
&=_q 2\sum_{n_1,\dots,n_k \geq 0} (q^2;q^2)_{n_k}\frac{q^{2n_1^2+2n_1+ \cdots +2n_{k-1}^2 + 2n_{k-1}}(q;q^2)_{n_1}}{(-q)_{2n_1+1}} \prod_{i=1}^{k-1} \begin{bmatrix} n_{i+1} \\ n_i \end{bmatrix}_{q^2}.
\end{aligned}
\end{equation*}
The symbol $=_q$ means that the two sides are equal at any odd root of unity (in which case the sums become finite) but not as functions inside the unit disk (where neither series converges, anyway).    The case $k=1$ reads
\begin{equation*}
\sum_{n \geq 0} \frac{(q)_n}{(-q)_n} =_q 2\sum_{n \geq 0} \frac{(q)_{2n}}{(-q)_{2n+1}}.
\end{equation*}
%- Other variants?   Bressoud?   G\"ollnitz-Gordon?   Note that we don't need to do $\rho_1, \rho_2 \to \infty$ every time...

%- Applications - generating functions for $L$-values, quantum modularity, congruences

%- Prove quantum Jacobi properties?    

%- Further examples?    

%- Strange identities of Robert et al?

Second, it would be nice to have strange identities for other classes of partial theta series of the form 
\begin{equation*}
\sum_{n \geq 0} n^{\nu}\chi(n)q^{\frac{n^2-a}{b}},
\end{equation*}  
where $\nu \in \{0,1\}$, $a \geq 0$ and $b >0$ are integers, and $\chi(n)$ is an appropriate periodic function.    One especially interesting class would be
\begin{equation*}
\sum_{n \geq 0} n\chi_{2st}(n)q^{\frac{n^2-(st-s-t)^2}{4st}},
\end{equation*}
where $\chi_{2st}(n)$ is the even periodic function modulo $2st$ defined by
\begin{equation*} 
\chi_{2st}(n) = 
\begin{cases}
1, &\text{if $n \equiv st-s-t$ or $st+s+t \pmod{2st}$}, \\
-1, &\text{if $n \equiv st-s+t$ or $st+s-t \pmod{2st}$}, \\
0, &\text{otherwise}.
\end{cases}
\end{equation*}
As explained in \cite{Hi-Ki1,Hi-Ki2}, this series captures the values of the colored Jones polynomial for the torus knot $T_{s,t}$ at roots of unity.     The case $T_{2,2k+1}$ corresponds to Hikami's strange identities \eqref{Hikamistrange} when $a=0$.   The case of torus knots $T_{3,2^k}$ is studied in \cite{Bietal}, though the strange identities found there do not appear to be simple consequences of the Bailey pair framework presented here.  

Finally, as is often the case when identities are understood in the context of Bailey pairs, we have only scratched the surface of what is possible.   For instance, in establishing families of Bailey pairs like those in \eqref{Hikamialpha} and \eqref{Hikamibeta},  we have iterated the Bailey Lemma in \eqref{alphaprimedef} and \eqref{betaprimedef} only in the special case $\rho_1, \rho_2 \to \infty$.   This typically leads to the most elegant results, but there are other possibilities.   For example, if we follow the proof of \eqref{Hikamialpha} and \eqref{Hikamibeta} but use $\rho_1 = -q$, $\rho_2 \to \infty$ in the very last step, and then follow the steps in the proof of Hikami's strange identities with appropriate modifications, we find the strange identities
\begin{equation} \label{family5} 
\begin{aligned}
\sum_{n_1,\dots,n_k \geq 0}& \frac{(q)_{n_k}(-q)_{n_{k-1}}q^{n_1^2+ \cdots + n_{k-2}^2+ \binom{n_{k-1}+1}{2} + n_{a+1} + \cdots + n_{k-2}}}{(-q)_{n_k}} \prod_{i=1}^{k-1} \begin{bmatrix} n_{i+1} + \delta_{i,a} \\ n_i \end{bmatrix}  \\
&``=" - \sum_{n \geq 0} n\chi_{4k}^{(a)}(n)q^{\frac{n^2 - (k-a-1)^2}{4k}}.
\end{aligned}
\end{equation}
Here  $0 \leq a < k-1$, and $\chi_{4k}^{(a)}(n)$ is the even periodic function modulo $4k$ defined by
\begin{equation*} 
\chi_{4k}^{(a)}(n) = 
\begin{cases}
1, &\text{if $n \equiv k-a-1$ or $3k+a+1 \pmod{4k}$}, \\
-1, &\text{if $n \equiv k+a+1$ or $3k-a-1 \pmod{4k}$}, \\
0, &\text{otherwise}.
\end{cases}
\end{equation*} 
Comparing \eqref{family5} with \eqref{family3} and \eqref{Hikamistrange} gives two families of quantum $q$-series identities, 
\begin{equation*}
\begin{aligned}
\sum_{n_1,\dots,n_{2k} \geq 0}& \frac{(q)_{n_{2k}}(-q)_{n_{2k-1}}q^{n_1^2+ \cdots + n_{2k-2}^2+ \binom{n_{2k-1}+1}{2} + n_{2a+1} + \cdots + n_{2k-2}}}{(-q)_{n_{2k}}} \prod_{i=1}^{2k-1} \begin{bmatrix} n_{i+1} + \delta_{i,2a} \\ n_i \end{bmatrix} \\
&=_q 2 \sum_{n_1,\dots,n_k \geq 0} (q^2;q^2)_{n_k}\frac{q^{2n_1^2+ \cdots +2n_{k-1}^2 + 2n_{a+1} + \cdots + 2n_{k-1}}}{(-q;q^2)_{n_1+\delta_{a,0}}} \prod_{i=1}^{k-1} \begin{bmatrix} n_{i+1}  + \delta_{i,a}\\ n_i \end{bmatrix}_{q^2}
\end{aligned}
\end{equation*}   
and
\begin{equation*}
\begin{aligned}
\sum_{n_1,\dots,n_{2k+1} \geq 0}& \frac{(q)_{n_{2k+1}}(-q)_{n_{2k}}q^{n_1^2+ \cdots + n_{2k-1}^2+ \binom{n_{2k}+1}{2} + n_{2a+2} + \cdots + n_{2k-1}}}{(-q)_{n_{2k+1}}} \prod_{i=1}^{2k} \begin{bmatrix} n_{i+1} + \delta_{i,2a+1} \\ n_i \end{bmatrix} \\
&=_q 2 \sum_{n_1,\dots,n_k \geq 0} (q^2;q^2)_{n_k}q^{2n_1^2+ \cdots +2n_{k-1}^2 + 2n_{a+1} + \cdots + 2n_{k-1}} \prod_{i=1}^{k-1} \begin{bmatrix} n_{i+1} + \delta_{i,a} \\ n_i \end{bmatrix}_{q^2},
\end{aligned}
\end{equation*} 
valid at odd roots of unity.
%The last identity is particularly noteworthy.   Hikami showed that the series on the right-hand side at $a=0$ captures the \emph{Kashaev invariant}, or value %of the $N$th colored Jones polynomial at $e^{2 \pi i / N}$, of the torus knot $T_{2,2k+1}$ \cite{Hi.5,Hi1}.   The same is therefore true of the left-hand side %for odd $N$.      
%%%%%%%%%%
%%% References 
%%%%%%%%%%%%

\end{document}